\documentclass[reqno]{amsart}

\usepackage[
    eientei,
    +maths,
	+tikz,
	margins = reduced,
    scriptStyle = euler,
]{kappak-amsart} %

\usepackage{bussproofs} %
\usepackage{catchfilebetweentags} %
\usepackage{enumitem} %
\usepackage{MnSymbol} %
\usepackage{pict2e} %
\usepackage{tikz-cd} %

\usetikzlibrary{calc, cd, decorations.pathmorphing, positioning}
\tikzset{treeedge/.style = {above, sloped}}
\tikzset{treenode/.style = {circle, minimum size = 3pt, inner sep = 0, draw, fill}}

\input{macros.tex}

\title{The equivalence between many-to-one polygraphs and opetopic sets}
\author{C\'edric Ho Thanh}

\address{Research Institute for Foundations of Computer Science (IRIF), Paris
University, Paris, France}


\begin{document}

\begin{abstract}
	From the polynomial approach to the definition of opetopes of Kock et al.,
	we derive a category of opetopes, and show that its set-valued presheaves,
	or opetopic sets, are equivalent to many-to-one polygraphs.	As an immediate
	corollary, we establish that opetopic sets are equivalent to multitopic
	sets, introduced and studied by Harnick et al, and we also address an open
	question of Henry.
\end{abstract}

\maketitle

\begin{small}
    \tableofcontents
\end{small}

\section{Introduction}

Opetopes were originally introduced by Baez and Dolan in \cite{Baez1998} as an
algebraic structure to describe compositions and coherence laws in weak higher
dimensional categories. They differ from other shapes (such as globular or
simplicial) by their (higher) tree structure, giving them the informal
designation of ``many-to-one''. Pasting opetopes give rise to opetopes of
higher dimension (it is in fact how they are defined!), and the analogy between
opetopes and cells in a free higher category starts to emerge. On the other
hand, polygraphs (also called computads) are higher dimensional directed graphs
used to generate free higher categories by specifying generators and the way
they may be pasted together (by means of source and targets).

In this paper, we relate opetopes and polygraphs in a direct way. Namely, we
define a category $\bbOO$ whose objects are opetopes, in such a way that the
category of its $\Set$-valued presheaves, or opetopic sets, is equivalent to
the category of many-to-one polygraphs. This equivalence was already known from
\cite{Harnik2002, Harnik2008, Hermida2000}, however the proof is very indirect.
The recent work of Henry \cite{Henry2019} showed the category of many-to-one
polygraphs (among many others) to be a presheaf category, but left the
equivalence between ``opetopic plexes'' (serving as shapes for many-to-one
polygraphs in his paper) and opetopes open. We establish this in our present
work.

The notion of multitope \cite{Hermida2002, Harnik2008} is related to that of
opetope, and has been developed based on similar motivations. However the
approaches used are different: \emph{ope}topes are based on \emph{ope}rads
(specifically, $T_n$-operads, where $T_n$ is a certain sequence of cartesian
monads) \cite{Leinster2004}, while \emph{multi}topes are based on (symmetric)
\emph{multi}categories. It is known that multitopic sets are equivalent to
many-to-one polygraphs \cite{Harnik2008, Harnik2002}, and in particular our
present contribution reasserts the equivalence between multitopic and opetopic
sets.

\subsection*{Plan}

We begin by recalling elements of the theory of polynomial functors and
polynomial monads \cref{sec:polynomial-functors-and-monads}. This formalism is
at the base of our chosen approach to opetopes, which we present in
\cref{sec:opetopes}. In \cref{sec:polygraphs}, we review some basic polygraphs
theory, and pay special attention to those that are many-to-one. Finally, we
state and prove the equivalence between opetopic sets and many-to-one
polygraphs in \cref{sec:equialence}.

\subsection*{Acknowledgments}

I would like to thank my PhD advisors, Pierre-Louis Curien and Samuel Mimram,
for their kind attention and guidance. This project has received funding from
the European Union's Horizon 2020 research and innovation program under the
Marie Sklodowska--Curie grant agreement No 665850.

\section{Polynomial functors and polynomial monads}
\label{sec:polynomial-functors-and-monads}

We survey elements of the theory of polynomial functors, trees, and monads. For
more comprehensive references, see \cite{Kock2011, Gambino2013}.

\subsection{Polynomial functors}

\begin{definition}
    [{Polynomial functor \cite[paragraph 1.4]{Gambino2013}}]
    \label{def:polynomial-functor}
    A \emph{polynomial functor $P$} is a diagram in
    $\Set$ of the form
    \begin{equation}
        \label{eq:polynomial-functor}
        \polynomialfunctor{I}{E}{B}{J.}{s}{p}{t}
    \end{equation}
    We say that $P$ is a \emph{polynomial endofunctor}
    if $I = J$. In this case, we also say that
    $P$ is a \emph{polynomial functor over $I$}. We say that $P$ is
    \emph{finitary}if the fibres of $p : E
    \longrightarrow B$ are finite sets. We will always assume polynomial
    functors and endofunctors to be finitary.

    We use the following terminology for a polynomial functor $P$ as in
    \cref{eq:polynomial-functor}, which is motivated by the intuition that a
    polynomial functor encodes a multi-sorted signature of function symbols.
    The elements of $B$ are called the \emph{nodes}or
    \emph{operations}of $P$, and for every node $b$, the
    elements of the fibre $E (b) \eqdef p^{-1} (b)$ are called the
    \emph{inputs}of $b$. The elements of $I$ are called the
    \emph{input colors}or \emph{input
    sorts}of $P$, and the elements of $J$ are \emph{output
    colors}or \emph{output sort}. For every input $e$ of a
    node $b$, we denote its color by $s_e (b) \eqdef s (e)$.
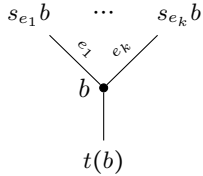
    \[
        \tikzinput{polynomial-functors}{operation}
    \]
\end{definition}

\begin{definition}
    [{Morphism of polynomial functor}]
    \label{def:polynomial-endofunctor-morphism}
    A morphism $f$ from a polynomial functor $P$ over $I$ (on the first row) to
    a polynomial functor $P'$ over $I'$ (on the second row) is a commutative
    diagram of the form
    \[
        \begin{tikzcd}
            I
                \ar[d, "f_0"'] &
            E
                \pullbackcorner \ar[r, "p"] \ar[d, "f_2" left]
                \ar[l, "s" above] &
            B
                \ar[r, "t"] \ar[d, "f_1" left] &
            I
                \ar[d, "f_0" left] \\
            I' &
            E'
                \ar[r, "p'"] \ar[l, "s'" above] &
            B'
                \ar[r, "t'"] &
            I'
        \end{tikzcd}
    \]
    where the middle square is cartesian (i.e. is a pullback square). If $P$
    and $P'$ are both polynomial functors over $I$, then a morphism from $P$ to
    $P'$ \emph{over $I$} is a commutative diagram as above, but where $f_0$ is
    required to be the identity \cite[paragraph 0.1.3]{Kock2011} \cite[section
    2.5]{Kock2010}. Let $\PolyEnd$ denote the
    category of polynomial functors and morphisms of polynomial functors, and
    $\PolyEnd (I)$ the category of polynomial
    functors over $I$ and morphisms of polynomial functors over $I$.
\end{definition}

\subsection{Trees}
\label{sec:trees}

The combinatorial notion of a tree fits nicely in the framework of polynomial
functors. We now state the definition of a polynomial tree, and refer the
reader to \cite[section 1.0.3]{Kock2011} for more details about the intuition
behind it.

\begin{definition}
    [{Polynomial tree \cite[section 1.0.3]{Kock2011}}]
    \label{def:polynomial-tree}
    A polynomial functor $T$ given by
    \[
        \polynomialfunctor{T_0}{T_2}{T_1}{T_0}{s}{p}{t}
    \]
    is a \emph{polynomial tree} (or just
    \emph{tree}) if
    \begin{enumerate}
        \item the sets $T_0$, $T_1$ and $T_2$ are finite (in particular, each
        node has finitely many inputs); by convention we assume $T_0 \neq
        \emptyset$;
        \item the map $t$ is injective;
        \item the map $s$ is injective, and the complement of its image $T_0 -
        \im s$ has a single element, called the \emph{root};
        \item let $T_0 = T_2 + \{ r \}$, with $r$ the root, and define the
        \emph{walk-to-root} function $\sigma$ by
        $\sigma (r) = r$, and otherwise $\sigma (e) = t p (e)$; then we ask
        that for all $x \in T_0$, there exists $k \in \bbNN$ such that
        $\sigma^k (x) = r$.
    \end{enumerate}
    We call the colors of a tree its \emph{edges} and the inputs
    of a node the \emph{input edges} of that node.

    Let $\Tree$ be the full
    subcategory of $\PolyEnd$ whose objects are trees. Note that it is the
    category of \emph{symmetric} or \emph{non-planar} trees (the automorphism
    group of a tree is in general non-trivial) and that its morphisms
    correspond to inclusions of non-planar subtrees. An \emph{elementary tree}
    is a tree with at most one node, and we
    write $\elTree$ for the full
    subcategory of $\Tree$ spanned by elementary trees.
\end{definition}

\begin{definition}
    [$P$-tree]
    \label{def:p-tree}
    For $P \in \PolyEnd$, the category $\tree P$ of
    \emph{$P$-trees} is the slice $\Tree / P$. If $f : P
    \longrightarrow Q$ is a morphism of polynomial functors, then it induces a
    natural functor $f_* : \tree P \longrightarrow \tree Q$ by postcomposition.
\end{definition}

\begin{notation}
    \label{not:p-tree}
    A $P$-tree $T \in \tree P$ is a morphism from a polynomial tree, which we
    shall denote by $\underlyingtree{T}$, to $P$, as in $T : \underlyingtree{T}
    \longrightarrow P$. We point out that $\underlyingtree{T}_1$ is the set of
    nodes of the $P$-tree $T$, while $T_1 : \underlyingtree{T}_1
    \longrightarrow P_1$ provides a \emph{decoration} of the nodes of
    $\underlyingtree{T}$ by operations of $P$, and likewise for edges.
\end{notation}

\begin{definition}
    [Address]
    \label{def:address}
    Let $T \in \Tree$ be a polynomial tree and $\sigma$ be its walk-to-root
    function (\cref{def:polynomial-tree}). We define the \emph{address}
    function $\addr$on edges inductively as
    follows:
    \begin{enumerate}
        \item if $r$ is the root edge, let $\addr r \eqdef []$,
        \item if $i \in T_0 - \{ r \}$ and if $\addr \sigma (i) = [x]$, define
        $\addr i \eqdef [x e]$, where $e \in T_2$ is the unique element such
        that $s(e) = i$.
    \end{enumerate}
    Thus an address is a sequence of elements of $T_2$, enclosed by brackets.
    The address of a node $b \in T_1$ is simply $\addr b \eqdef \addr t (b)$.
    Note that this function is injective since $t$ is. Let
    $T^\nodesymbol$denote its image,
    the set of \emph{node addresses}of $T$, and let
    $T^\leafsymbol$be the set of
    addresses of leaf edges, i.e. those edges not in the
    image of $t$.

    Assume now that $U : \underlyingtree{U} \longrightarrow P$ is a $P$-tree.
    If $b \in \underlyingtree{U}_1$ has address $\addr b = [p]$, write
    $\src_{[p]} U \eqdef U_1 (b)$. For convenience, we let $U^\nodesymbol
    \eqdef \underlyingtree{U}^\nodesymbol$, and $U^\leafsymbol \eqdef
    \underlyingtree{U}^\leafsymbol$.
\end{definition}

\begin{remark}
    The formalism of addresses is a useful bookkeeping syntax for the
    operations of grafting and substitution on trees. The syntax of addresses
    will extend to the category of opetopes and will allow us to give a precise
    description of the composition of morphisms in the category of opetopes
    (see \cref{def:o}) as well as certain constructions on opetopic
    sets.
\end{remark}

\begin{notation}
    \label{not:marked-tree}
    We denote by $\treewithleaf P$the set of $P$-trees with a marked leaf, i.e.
    endowed with the address of one of its leaves. Similarly, we denote by
    $\treewithnode P$ the set of $P$-trees with a marked node.
\end{notation}

\begin{definition}
    [Elementary $P$-trees]
    \label{def:elementary-p-trees}
    Let $P$ be a polynomial endofunctor as in equation
    \cref{eq:polynomial-functor}. For $i \in I$, define $\itree{i} \in \tr
    P$as having underlying tree
    \begin{equation}
        \label{eq:polynomial-functor:itree}
        \polynomialfunctor
            {\{i\}}{\emptyset}{\emptyset}{\{i \},}
            {}{}{}
    \end{equation}
    along with the obvious morphism to $P$, that which maps $i$ to $i \in I$.
    This corresponds to a tree with no nodes and a unique edge, decorated by
    $i$. Define $\ytree{b} \in \tr P$, the
    \emph{corolla}at $b$, as having underlying tree
    \begin{equation}
        \label{eq:polynomial-functor:ytree}
        \polynomialfunctor
            {s(E(b)) + \{*\}}{E (b)}{\{b \}}{s(E(b)) + \{*\},}
            {s}{}{}
    \end{equation}
    where the right map sends $b$ to $*$, and where the morphism $\ytree{b}
    \longrightarrow P$ is the identity on $s (E (b)) \subseteq I$, maps $*$ to
    $t (b) \in I$, is the identity on $E (b) \subseteq E$, and maps $b$ to $b
    \in B$. This corresponds to a $P$-tree with a unique node, decorated by
    $b$. Observe that for $T \in \tr P$, giving a morphism $\itree{i}
    \longrightarrow T$ is equivalent to specifying the address $[p]$ of an
    edge of $T$ decorated by $i$. Likewise, morphisms of the form
    $\ytree{b} \longrightarrow T$ are in bijection with addresses of nodes
    of $T$ decorated by $b$.
\end{definition}

\begin{remark}
    \label{rem:elementary-p-trees:addresses}
    Let $P$ be a polynomial endofunctor as in \cref{eq:polynomial-functor}.
    \begin{enumerate}
        \item Let $i \in I$ be a color of $P$. Since $\itree{i}$ does not have
        any nodes, the set $\itree{i}^\nodesymbol$ of its node addresses is
        empty. On the other hand, the set of its leaf addresses is
        $\itree{i}^\leafsymbol = \left\{ [] \right\}$, since the unique leaf is
        the root edge.
        \item Let $b \in B$ be an operation of $P$. Then $\ytree{b}^\nodesymbol
        = \left\{ [] \right\}$ since the only node is that above the rood edge.
        For leaves, we have $\ytree{b}^\leafsymbol = \left\{ [e] \mid e \in E
        (b) \right\}$.
    \end{enumerate}
\end{remark}

\begin{definition}
    [Grafting]
    \label{def:grafting}
    For $S, T \in \tr P$, $[l] \in S^\leafsymbol$ such that the
    leaf of $S$ at $[l]$ and the root edge of $T$ are decorated by the
    same $i \in I$, define the \emph{grafting}$S
    \graft_{[l]} T$ of
    $S$ and $T$ on $[l]$ by the following pushout (in $\tree P$):
    \begin{equation}
        \label{eq:polynomial-functor:grafting}
        \pushoutdiagram
            {\itree{i}}{T}{S}{S \graft_{[l]} T .}
            {[]}{[l]}{}{}
    \end{equation}
    Note that if $S$ (resp. $T$) is a trivial tree, then $S \graft_{[l]} T = T$
    (resp. $S$). We assume, by convention, that the grafting operator $\circ$
    associates to the right. In particular $\itree{i} \graft_{[]} T = T$ and $S
    \graft_{[l]} \itree{i} = S$.
\end{definition}

\begin{lemma}
    \label{lemma:grafting-addresses}
    For $S, T \in \tr P$, $[l] \in S^\leafsymbol$ such that the grafting $S
    \graft_{[l]} T$ is defined, we have
    \[
        (S \graft_{[l]} T)^\nodesymbol \speq S^\nodesymbol +
            \left\{ [lp] \mid [p] \in T^\nodesymbol \right\} ,
        \qqquad
        (S \graft_{[l]} T)^\leafsymbol \speq S^\leafsymbol - \left\{ [l] \right\} +
            \left\{ [lp] \mid [p] \in T^\leafsymbol \right\} .
    \]
\end{lemma}

\begin{notation}
    [Total grafting]
    \label{not:total-grafting}
    Let $T, U_1, \ldots, U_k \in \tr P$, write $T^\leafsymbol = \left\{ [l_1],
    \ldots, [l_k] \right\}$, and assume the grafting $T \graft_{[l_i]} U_i$ is
    defined for all $i$. Then the \emph{total grafting} will be denoted
    concisely by
    \begin{equation}
        \label{eq:big-grafting}
        T \biggraft_{[l_i]} U_i
        = ( \cdots
            (T \graft_{[l_1]} U_1) \graft_{[l_2]} U_2
        \cdots ) \graft_{[l_k]} U_k .
    \end{equation}
    It is easy to see that the result does not depend on the order in which the
    graftings are performed.
\end{notation}

\begin{proposition}
    [{\cite[proposition 1.1.21]{Kock2011}}]
    \label{prop:polynomial-functor:trees-are-graftings}
    Every $P$-tree is either of the form $\itree{i}$, for some $i \in I$, or
    obtained by iterated graftings of corollas (i.e. $P$-trees of the form
    $\ytree{b}$ for $b \in B$).
\end{proposition}
\begin{proof}
    This can easily be proved by induction on the number of nodes.
\end{proof}

\begin{remark}
    \label{rem:tree-contexts}
    As a consequence of \cite[proposition 1.1.3]{Kock2011}, a morphism $T
    \longrightarrow S$ of $P$-trees exhibits $T$ as a subtree of $S$ as in
    \[
        S \speq U \graft_{[p]} T \biggraft_{[l]} V_{[l]}
    \]
    where $U$ is spanned by all the edges of $S$ that are either descendant of
    the root edge of $T$, or incomparable to it \cite[paragraphs 1.0.7 and
    1.1.11]{Kock2011}, and where $[l]$ ranges over $T^\leafsymbol$. Conversely,
    any such decomposition of $S$ induces a morphism $T \longrightarrow S$.
\end{remark}

\subsection{Polynomial monads}
\label{sec:polynomial:monads}

\begin{definition}
    [Polynomial monad]
    \label{def:polynomial-monad}
    A \emph{polynomial monad over $I$}is a monoid in
    $\PolyEnd(I)$. Note that a polynomial monad over $I$ is thus necessarily a
    cartesian monad on $\Set/I$.\footnote{We recall that a monad is
    \emph{cartesian}if its endofunctor preserves
    pullbacks and its unit and multiplication are cartesian natural
    transformations.} Let $\PolyMnd(I)$be the category of
    monoids in $\PolyEnd(I)$. That is, $\PolyMnd(I)$ is the category of
    polynomial monads over $I$ and morphisms of polynomial functors over $I$
    that are also monad morphisms.
\end{definition}

\begin{definition}
    [$(-)^\star$ construction]
    \label{def:star-construction}
    Given a polynomial endofunctor $P$ as in \cref{eq:polynomial-functor},
    we define a new polynomial endofunctor $P^\star$ as
    \begin{equation}
        \label{eq:free-monad}
        \polynomialfunctor
            {I}{\treewithleaf P}{\tree P}{I}
            {s}{p}{t}
    \end{equation}
    where $s$ maps a $P$-tree with a marked leaf to the decoration of that
    leaf, $p$ forgets the marking, and $t$ maps a $P$-tree to the decoration of
    its root. Remark that for $T \in \tree P$ we have $p^{-1} (T) \cong
    T^\leafsymbol$. Clearly, there is an inclusion $P \longrightarrow P^\star$,
    mapping $b \in B$ to $\ytree{b} \in \tree P$, and $e \in E (b)$ to $[e] \in
    \ytree{b}^\leafsymbol$ (see \cref{rem:elementary-p-trees:addresses}).
\end{definition}

\begin{theorem}
    [{\cite[section 1.2.7]{Kock2011}, \cite[sections 2.7 to 2.9]{Kock2010}}]
    \label{th:polynomial-monads-star-algebras}
    The polynomial functor $P^\star$ has a canonical structure of a polynomial
    monad. Furthermore, the functor $(-)^\star$ is left adjoint to the
    forgetful functor $\PolyMnd (I) \longrightarrow \PolyEnd (I)$, and the
    adjunction is monadic.
\end{theorem}

\begin{definition}
    [Target and readdressing map]
    \label{def:readdressing}
    We abuse notations and letting $(-)^\star$ denote the associated monad on
    $\PolyEnd (I)$. Let $M$ be a polynomial monad as in
    \[
        \polynomialfunctor{I}{E}{B}{I.}{s}{p}{t}
    \]
    By \cref{th:polynomial-monads-star-algebras}, $M$ is a $(-)^\star$-algebra,
    and we will write its structure map $M^\star \longrightarrow M$ as
    \begin{equation}
        \label{eq:polynomial-monad-structure-map}
        \begin{tikzcd}
            I
                \ar[d, equal] &
            \treewithleaf M
                \pullbackcorner
                \ar[l] \ar[d, "\readdress"] \ar[r] &
            \tree M
                \ar[r] \ar[d, "\tgt"] &
            I
                \ar[d, equal]
            \\
            I &
            E
                \ar[l] \ar[r] &
            B \ar[r] &
            I.
        \end{tikzcd}
    \end{equation}
    For $T \in \tree M$, we call $\readdress_T : T^\leafsymbol \xto{\cong} E
    (\tgt T)$ the \emph{readdressing} function of $T$, and
    $\tgt T \in B$ is called the \emph{target} of $T$. If we think of an
    element $b \in B$ as the corolla $\ytree{b}$, then the target map $\tgt$
    ``contracts'' a tree to a corolla, and since the middle square is a
    pullback, the number of leaves is preserved. The map $\readdress_T$
    establishes a decoration-preserving correspondence between the set
    $T^\leafsymbol$ of leaf addresses of a tree $T$ and the elements of $E
    (\tgt T)$.
\end{definition}

\begin{definition}
    [Baez--Dolan $(-)^+$ construction]
    \label{def:baez-dolan-construction}
    Let $M$ be a polynomial monad as in \cref{eq:polynomial-functor}, and
    define its \emph{Baez--Dolan construction}
    $M^+$ to be
    \begin{equation}
        \label{eq:polynomial-functor:+}
        \polynomialfunctor{B}{\treewithnode M}{\tree M}{B}{\src}{p}{\tgt}
    \end{equation}
    where $\src$ maps an $M$-tree with a marked node to the label of that node,
    $p$ forgets the marking, and $\tgt$ is the target map of
    \cref{def:readdressing}. If $T \in \tree M$, remark that $p^{-1} T
    = T^\nodesymbol$ is the set of node addresses of $T$. If $[p] \in
    T^\nodesymbol$, then $\src [p] \eqdef \src_{[p]} T$.
\end{definition}

\begin{theorem}
	[{\cite[section 3.2]{Kock2010}}]
    \label{th:polynomial-functor:+:is-monad}
    If $M$ a polynomial monad, then $M^+$ has a canonical structure of a
    polynomial monad.
\end{theorem}

\begin{remark}
    [Nested addresses]
    \label{rem:nested-addresses}
    Let $M$ be a polynomial monad, and $T \in \tree M^+$. Then the nodes of
    $T$ are decorated in $M$-trees, and its edges by operations of $M$.
    Assume that $U \in \tree M$ decorates some node of $T$, say $U
    = \src_{[p]} T$ for some node address $[p] \in T^\nodesymbol$.
    \begin{enumerate}
        \item The input edges of that node are in bijection with
        $U^\nodesymbol$. In particular, the address of those input edges
        are of the form $[p[q]]$, where $[q]$ ranges over $U^\nodesymbol$.
        This really motivates enclosing addresses in brackets.
        \item On the other hand, the output edge of that node is decorated by
        $\tgt U$ (where $\tgt$ is defined in \cref{def:readdressing}).
    \end{enumerate}
\end{remark}

\begin{notation}
    \label{not:edg}
    Let $M$ be a polynomial monad, and $T \in \tree M^+$. For $[a]$ the address
    of an edge of $T$, let $\edg_{[a]} T$be the operation
    of $M$ decorating that edge. Explicitely, if $[a] = []$, then $\edg_{[]} T
    \eqdef \tgt \src_{[]} T$. Otherwise, $[a] = [p[q]]$ for some $[p] \in
    T^\nodesymbol$ (the node below the edge) and $[q] \in (\src_{[p]}
    T)^\nodesymbol$, and let $\edg_{[p[q]]} T \eqdef \src_{[q]} \src_{[p]} T$.
\end{notation}

\section{Opetopes} \label{sec:opetopes}

\subsection{Definition}
\label{sec:opetopes:definition}

\begin{definition}
    [The $\optPolyFun^n$ monad]
    \label{def:zn}
    Let $\optPolyFun^0$ be the identity polynomial monad on $\Set$, as depicted
    on the left below, and let $\optPolyFun^n \eqdef
    (\optPolyFun^{n-1})^+$. Write $\optPolyFun^n$ as on
    right:
    \begin{equation}
        \label{eq:zn}
        \polynomialfunctor
            {\{* \}}{\{* \}}{\{* \}}{\{* \},}
            {}{}{}
        \qqquad
        \polynomialfunctor
            {\bbOO_n}{E_{n+1}}{\bbOO_{n+1}}{\bbOO_n .}
            {\src}{p}{\tgt}
    \end{equation}
\end{definition}

\begin{definition}
    [Opetope]
    \label{def:opetope}
    An \emph{$n$-dimensional opetope} (or \emph{$n$-opetope} for short) $\omega$
    is simply an element of $\bbOO_n$, and we write $\dim \omega = n$. If $n
    \geq 2$, then $n$-opetopes are exactly the $\optPolyFun^{n-2}$-trees. In
    this case, an opetope $\omega \in \bbOO_n$ is called \emph{degenerate}
    if its underlying tree has no nodes (and thus
    consists of a unique edge), so that $\omega = \itree{\phi}$ for some $\phi
    \in \bbOO_{n-2}$. We say that $\omega$ it is an \emph{endotope}
    it its underlying tree has exactly one node, i.e. $\omega
    = \ytree{\psi}$ for some $\psi \in \bbOO_{n-1}$.

    Following \cref{eq:polynomial-monad-structure-map}, for $n \geq 2$ and
    $\omega \in \bbOO_n$, the structure of polynomial monad
    $(\optPolyFun^{n-2})^\star \longrightarrow \optPolyFun^{n-2}$ gives a
    bijection $\readdress_\omega : \omega^\leafsymbol \longrightarrow (\tgt
    \omega)^\nodesymbol$ between the leaves of $\omega$ and the nodes of $\tgt
    \omega$, preserving the decoration by $(n-2)$-opetopes.
\end{definition}

\begin{example}
    \label{ex:opetopes}
    \begin{enumerate}
        \item The unique $0$-opetope is denoted $\optZero$ and called the
        \emph{point}.

        \item The unique $1$-opetope is denoted $\optOne$ and called the
        \emph{arrow}.

        \item If $n \geq 2$, then $\omega \in \bbOO_n$ is a
        $\optPolyFun^{n-2}$-tree, i.e. a tree whose nodes are labeled in
        $(n-1)$-opetopes, and edges are labeled in $(n-2)$-opetopes. In
        particular, $2$-opetopes are $\optPolyFun^0$-trees, i.e. linear trees,
        and thus in bijection with $\bbNN$. We will refer to them as
        \emph{opetopic integers}, and write
        $\optInt{n}$ for the
        $2$-opetope having exactly $n$ nodes.

        \item A $3$-opetope is a $\optPolyFun^1$-tree, i.e. a planar tree.

        \item A $4$-opetope is a $\optPolyFun^2$-tree. Unfolding definitions,
        if $\omega : \underlyingtree{\omega} \longrightarrow \optPolyFun^2$,
        then nodes of $\omega$ are decorated by elements of $\bbOO_3$, i.e.
        planar trees. Further, if $x \in \underlyingtree{\omega}_1$ is a node
        of $\omega$, then $\omega_2$ exhibits a bijection between the input
        edges of $x$ and the nodes of $\omega_1 (x) \in \bbOO_3$.
    \end{enumerate}
\end{example}

\subsection{The category of opetopes}

Akin to the work of Cheng \cite{Cheng2003}, we define a category of opetopes by
means of generators and relations. The difference with the aforementioned
reference is our use of polynomial opetopes (also equivalent to Leinster's
definition \cite{Leinster2004, Kock2010}), while Cheng uses an approach by
multicategorical slicing, yielding ``symmetric'' opetopes.

\begin{lemma}
    [Opetopic identities]
    \label{lemma:opetopic-identities}
    Let $\omega \in \bbOO_n$ with $n \geq 2$.
    \begin{enumerate}
        \item (Inner edge) For an inner edge $[p[q]] \in \omega^\nodesymbol$
        (the fact that $\omega$ has an inner edge implies that it is non
        degenerate), we have $\tgt \src_{[p[q]]} \omega = \src_{[q]} \src_{[p]}
        \omega$.
        \item (Globularity 1) If $\omega$ is non degenerate, we have $\tgt
        \src_{[]} \omega = \tgt \tgt \omega$.
        \item (Globularity 2) If $\omega$ is non degenerate, and $[p[q]] \in
        \omega^\leafsymbol$, we have $\src_{[q]} \src_{[p]} \omega =
        \src_{\readdress_\omega [p[q]]} \tgt \omega$.
        \item (Degeneracy) If $\omega$ is degenerate, we have $\src_{[]} \tgt
        \omega = \tgt \tgt \omega$.
    \end{enumerate}
\end{lemma}
\begin{proof}
    \begin{enumerate}
        \item (Inner edge) By definition of a $\optPolyFun^{n-2}$-tree.

        \item (Globularity 1 and 2) By
        \cref{th:polynomial-monads-star-algebras}, the monad structure on
        $\optPolyFun^{n-2}$ amounts to a structure map
        $(\optPolyFun^{n-2})^\star \longrightarrow \optPolyFun^{n-2}$, which,
        taking the notations of \cref{def:readdressing}, is written as
        \[
            \begin{tikzcd}
                \bbOO_{n-2}
                    \ar[d, equal] &
                \treewithleaf \optPolyFun^{n-2}
                    \pullbackcorner
                    \ar[r, "p"]
                    \ar[d, "\readdress" left]
                    \ar[l, "\edg" above] &
                \tree \optPolyFun^{n-2}
                    \ar[r, "\edg_{[]}"]
                    \ar[d, "\tgt" left] &
                \bbOO_{n-2}
                    \ar[d, equal] \\
                \bbOO_{n-2} &
                \bbOO^\nodesymbol_{n-1}
                    \ar[r, "p"]
                    \ar[l, "\src" above] &
                \bbOO_{n-1}
                    \ar[r, "\tgt"] &
                \bbOO_{n-2} .
            \end{tikzcd}
        \]
        The claims follow from the commutativity of the right and left square
        respectively.

        \item (Degeneracy) Let $\omega = \itree{\phi}$, for $\phi \in
        \bbOO_{n-2}$. Then, $\tgt \tgt \omega = \tgt \ytree{\phi} = \phi$, and
        clearly, $\phi = \src_{[]} \ytree{\phi} = \src_{[]} \tgt \omega$.
        \qedhere
    \end{enumerate}
\end{proof}

\begin{definition}
    [The category $\bbOO$ of opetopes]
    \label{def:o}
    With the identities of \cref{lemma:opetopic-identities}, we define the
    category $\bbOO$ of opetopes by generators and relations as follows.
    \begin{enumerate}
        \item (Object) We set $\ob \bbOO = \sum_{n \in \bbNN} \bbOO_n$.
        \item (Generating morphism) Let $\omega \in \bbOO_n$ with $n \geq 1$.
        We introduce a generator $\tgt : \tgt \omega \longrightarrow
        \omega$, called the \emph{target
        embedding}. If $[p] \in \omega^\nodesymbol$,
        then we introduce a generator $\src_{[p]} : \src_{[p]} \omega
        \longrightarrow \omega$,
        called a \emph{source embedding}. An
        \emph{elementary face embedding}is
        either a source or the target embedding. \footnote{Note if $n \geq 2$
        and $[p]$ is an edge address of $\omega$, then the edge operation of
        $\edg_{[p]}$ \cref{not:edg} also translates to a morphism $\edg_{[p]} :
        \edg_{[p]} \omega \longrightarrow \omega$.}
        \item (Relation) We impose 4 relations described by the following
        commutative squares, which just enforce the identities of
        \cref{lemma:opetopic-identities}. Let $\omega \in \bbOO_n$ with $n \geq
        2$.
        \begin{itemize}
            \item \condition{Inner} For
            $[p[q]] \in \omega^\nodesymbol$ (forcing $\omega$ to be non
            degenerate), the following square must commute:
            \[
                \squarediagram
                    {\src_{[q]} \src_{[p]} \omega}{\src_{[p]} \omega}
                        {\src_{[p[q]]}\omega}{\omega}
                    {\src_{[q]}}{\tgt}{\src_{[p]}}{\src_{[p[q]]}}
            \]
            \item \condition{Glob1} If
            $\omega$ is non degenerate, the following square must commute:
            \[
                \squarediagram
                    {\tgt \tgt \omega}{\tgt \omega}{\src_{[]} \omega}{\omega .}
                    {\tgt}{\tgt}{\tgt}{\src_{[]}}
            \]
            \item \condition{Glob2} If
            $\omega$ is non degenerate, and for $[p [q]] \in
            \omega^\leafsymbol$, the following square must commute:
            \[
                \diagramsize{2}{4}
                \squarediagram
                    {\src_{\readdress_\omega [p [q]]} \tgt \omega}{\tgt \omega}
                        {\src_{[p]} \omega}{\omega .}
                    {\src_{\readdress_\omega [p[q]]}}{\src_{[q]}}{\tgt}
                        {\src_{[p]}}
            \]
            \item \condition{Degen} If
            $\omega$ is degenerate, the following square must commute:
            \[
                \squarediagram
                    {\tgt \tgt \omega}{\tgt \omega}{\tgt \omega}{\omega .}
                    {\tgt}{\src_{[]}}{\tgt}{\tgt}
            \]
        \end{itemize}
    \end{enumerate}
\end{definition}

\begin{remark}
    \label{rem:o}
    Let us explain this definition a little more. Opetopes are trees whose
    nodes (and edges) are decorated by opetopes. The decoration is now
    interpreted as a geometrical feature, namely as an embedding of a lower
    dimensional opetope. Further, the target of an opetope, while not an
    intrinsic data, is also represented as an embedding. The relations can be
    understood as follows.
    \begin{itemize}
        \item \condition{Inner} The inner edge at $[p[q) \in
        \omega^\nodesymbol$ is decorated by the target of the decoration of the
        node ``above'' it (here $\src_{[p[q]]} \omega$), and in the
        $[q]$-source of the node ``below'' it (here $\src_{[p]} \omega$). By
        construction, those two decorations match, and this relation makes the
        two corresponding embeddings $\src_{[q]} \src_{[p]} \omega
        \longrightarrow \omega$ match as well. On the left is an informal
        diagram about $\omega$ as a tree (reversed gray triangle), and on the
        right is an example of pasting diagram represented by an opetope, with
        the relevant features of the \condition{Inner} relation colored or
        thickened.
        \[
            \tikzinput[.9]{optids-informal}{inner}
            \qqquad
            \tikzinput[.9]{optids-informal}{inner.ps}
        \]
        \item \condition{Glob1-2} If we consider the underlying tree of
        $\omega$ as its ``geometrical source'', and the corolla $\ytree{\tgt
        \omega}$ as its ``geometrical target'', then they should be parallel.
        The relation \condition{Glob1} expresses this idea by ``gluing'' the
        root edges of $\omega$ and $\ytree{\tgt \omega}$ together, while
        \condition{Glob2} glues the leaves according to $\readdress_\omega$.
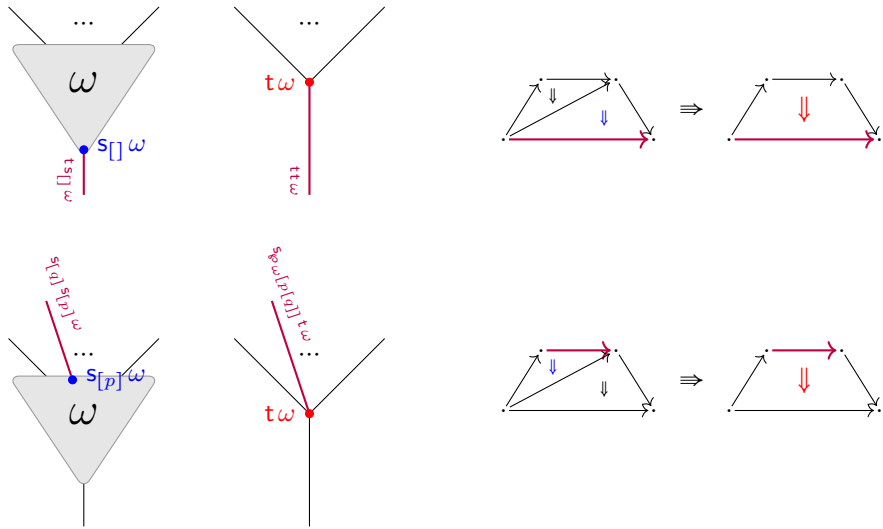
        \[
            \tikzinput[.9]{optids-informal}{glob1}
            \qqquad
            \tikzinput[.9]{optids-informal}{glob1.ps}
        \]
        \[
            \tikzinput[.9]{optids-informal}{glob2}
            \qqquad
            \tikzinput[.9]{optids-informal}{glob2.ps}
        \]
        \item \condition{Degen} If $\omega$ is a degenerate opetope, depicted
        as on the right, then its target should be a ``loop'', i.e. its only
        source and its target should be glued together.
        \[
            \tikzinput{optids-informal}{degen}
            \qqquad\qqquad\qqquad
            \tikzinput{optids-informal}{degen.ps}
        \]
    \end{itemize}
\end{remark}

\begin{notation}
    \label{not:subcategories-of-o}
    For $n \in \bbNN$, let $\bbOO_{\leq n}$ be the full subcategory of $\bbOO$
    spanned by opetopes of dimension at most $n$.
\end{notation}

\section{Opetopic sets and many-to-one polygraphs}
\label{sec:polygraphs}

Polygraphs where originally introduced by Street \cite[section 2]{Street1976}
under the name of \emph{computad}. They are to (strict) $\omega$-categories
what graphs are to $1$-categories: a combinatorial device that freely generates
them. However, unlike graphs, the category $\Pol$ of polygraphs fails to be a
presheaf category \cite{Carboni2004} \cite{Makkai2008} \cite{Cheng2013}. The
obstruction for it to be the case is an unpleasant corollary of the
\emph{exchange law}: if $f$ and $g$ are endomorphisms of an identity cell, then
$fg = gf$.
\begin{align*}
    \tikzinput[.8]{polygraph}{1}
    &\quad\longleadsto\quad \tikzinput[.8]{polygraph}{2}
        \quad\longleadsto\quad \tikzinput[.8]{polygraph}{3} \\
    &\quad\longleadsto\quad \tikzinput[.8]{polygraph}{4}
        \quad\longleadsto\quad \tikzinput[.8]{polygraph}{5}
\end{align*}
Nonetheless, the recent work of Henry \cite{Henry2019} characterized many
subcategories of $\Pol$ to be presheaf categories. Among them, the category
$\PolMTO$ of \emph{many-to-one} polygraphs, in which the target (or codomain)
of generating cells are themselves generating cells. In this section, we relate
opetopes and many-to-one polygraphs in a formal way. Namely, we construct an
equivalence of categories $\polyreal{-} : \PshO \longrightarrow \PolMTO$,
called \emph{polygraphic realization}.

Recall from \cite[theorem 1]{Borceux1986} that if $\catAA$ and $\catBB$ are two
Cauchy-complete categories such that $\Psh\catAA \simeq \Psh\catBB$, then
$\catAA \simeq \catBB$ (see \cref{th:cauchy-complete} for more details). In
particular, any Cauchy-complete category $\catAA$ that acts as a ``shape theory
for many-to-one polygraphs'', i.e. such that $\Psh\catAA \simeq \PolMTO$, is
equivalent to $\bbOO$ (since it doesn't have any non-identity endomorphism, it
is Cauchy-complete). This shows that the geometrical intuition behind the
definition of $\bbOO$ (\cref{rem:o}) is essentially the unique way to
faithfully implement the combinatorics of pasting diagrams.

The fact that many-to-one polygraphs are equivalent to opetopic sets was
already known from \cite{Hermida2000} \cite{Harnik2002} \cite{Cheng2004}
\cite{Harnik2008}, however the proof there is indirect and spanned over
multiple articles. The formalism we developed so far allows us to establish
this result directly.

\subsection{Strict higher categories}

\begin{definition}
    [$\omega$-category]
    \label{def:omega-category}
    An \emph{$\omega$-category} $\catCC$ (also called
    \emph{strict $\infty$-category}) is the
    datum of a diagram of sets
    \[
        \begin{tikzcd}
            \catCC_0
                \ar[r, bend right, "\id" below]
            &
            \catCC_1
                \ar[l, "{\src, \tgt}" above]
                \ar[r, bend right, "\id" below]
            &
            \cdots
                \ar[l, "{\src, \tgt}" above]
                \ar[r, bend right, "\id" below]
            &
            \catCC_n
                \ar[l, "{\src, \tgt}" above]
                \ar[r, bend right, "\id" below]
            &
            \cdots
                \ar[l, "{\src, \tgt}" above]
        \end{tikzcd}
    \]
    with \emph{composition maps} $\graft_k : \catCC_{n, k} \longrightarrow
    \catCC_n$, where $k < n$ and $\catCC_{n, k}$ is the pullback
    \[
        \pullbackdiagram
            {\catCC_{n, k}}{\catCC_n}{\catCC_n}{\catCC_k,}
            {}{}{\tgt^{n-k}}{\src^{n-k}}
    \]
    such that the following conditions hold:
    \begin{enumerate}
        \item for all $k < n$, the diagram
        \[
            \begin{tikzcd} [column sep = 5em]
                \catCC_k
                    \ar[r, bend right, "\id^{n-k}" below]
                &
                \catCC_n
                    \ar[l, "{\src^{n-k}, \tgt^{n-k}}" above]
            \end{tikzcd}
        \]
        with the composition map $\graft_k : \catCC_{n, k} \longrightarrow
        \catCC_n$ is a $1$-category;
        \item for all $l < k < n$, the diagram
        \[
            \begin{tikzcd} [column sep = 5em]
                \catCC_l
                    \ar[r, bend right, "\id^{k-l}" below]
                &
                \catCC_k
                    \ar[l, "{\src^{k-l}, \tgt^{k-l}}" above]
                    \ar[r, bend right, "\id^{n-k}" below]
                &
                \catCC_n
                    \ar[l, "{\src^{n-k}, \tgt^{n-k}}" above]
            \end{tikzcd}
        \]
        with the composition maps $\graft_l : \catCC_{k, l} \longrightarrow
        \catCC_k$, $\graft_l : \catCC_{n, l} \longrightarrow \catCC_n$ and
        $\graft_k : \catCC_{n, k} \longrightarrow \catCC_n$ is a strict
        $2$-category.
    \end{enumerate}

    The maps $\src$ and $\tgt$ are called \emph{source} and \emph{target} maps,
    respectively, and if $x \in \catCC_k$, then $\id_x \eqdef \id (x)$ is the
    \emph{identity cell of $x$}. Note that by definition, the following
    equalities hold:
    \[
        \src \src x \speq \src \tgt x,
        \qqquad
        \tgt \src x \speq \tgt \tgt x,
        \qqquad
        \src \id_x \speq x \speq \tgt \id_x.
    \]
    The first two are called the \emph{globular identities}. Still by definition, for $0 \leq l < k < n$ and $w, x, y, z
    \in \catCC_n$ the following \emph{exchange law} holds:
    \[
        (w \graft_k x) \graft_l (y \graft_k z)
        \speq (w \graft_l y) \graft_k (x \graft_l z) ,
    \]
    assuming both sides are well-defined. Note that a strict $n$-category
    $\catCC$ may be viewed as an $\omega$-category where $\catCC_m$ only has
    identities for all $m > n$. Given an $\omega$-category $\catCC$, we write
    $\catCC_{\leq n}$ for the underlying strict
    $n$-category
    \[
        \begin{tikzcd}
            \catCC_0
                \ar[r, bend right, "\id" below]
            &
            \catCC_1
                \ar[l, "{\src, \tgt}" above]
                \ar[r, bend right, "\id" below]
            &
            \cdots
                \ar[l, "{\src, \tgt}" above]
                \ar[r, bend right, "\id" below]
            &
            \catCC_n .
                \ar[l, "{\src, \tgt}" above]
        \end{tikzcd}
    \]

    An \emph{$\omega$-functor} $f : \catBB
    \longrightarrow \catCC$ between two $\omega$-categories is just a sequence
    of maps $f_n : \catBB_n \longrightarrow \catCC_n$ that induces a
    $n$-functor $f_{\leq n} : \catBB_{\leq n} \longrightarrow \catCC_{\leq n}$
    for all $n \in \bbNN$. If the context is clear, we simply write $f$ for
    $f_n : \catBB_n \longrightarrow \catCC_n$.
\end{definition}

\begin{definition}
    [{Parallel cells \cite{Metayer2003}}]
    \label{def:parallel-cells-hcat}
    Let $\catDD$ be a strict $\omega$-category, $n \in \bbNN$. Two $n$-cells
    $x, y \in \catDD_n$ are \emph{parallel}, denoted by $x
    \parallel y$, if $\src x = \src y$
    and $\tgt x = \tgt y$. By convention, $0$-cells are pairwise parallel.
\end{definition}

\begin{definition}
    [Cellular extension] Let $\catDD$ be a strict $(n-1)$-category. A
    \emph{cellular extension} of $\catDD$ consists
    in a set $X$ and two maps $\src, \tgt : X \longrightarrow \catDD_{n-1}$
    such that the \emph{globular identities} hold,
    i.e. for all $x \in X$, we have $\src x \parallel \tgt x$. We also denote
    such a cellular extension by
    \[
        \catDD \xot{\src, \tgt} X .
    \]
\end{definition}

\begin{definition}
    [Free $n$-category] Let $\catDD$ be a strict $(n-1)$-category and $\catDD
    \xot{\src, \tgt} X$ be a cellular extension of $\catDD$. The \emph{free
    strict $n$-category} generated by this cellular extension is the strict
    $n$-category $\catDD [X]$ such that
    \begin{enumerate}
        \item as strict $(n-1)$-categories, $\catDD = \catDD [X]_{\leq n-1}$;
        \item there is an inclusion $X \longhookrightarrow \catDD [X]_n$, and
        the following diagrams commute:
        \[
            \diagramarrows{<-}{c->}{<-}{}
            \triangleDRdiagram
                {X}{\catDD_{n-1}}{\catDD [X]_n,}
                {\src}{}{\src}
            \qqquad
            \triangleDRdiagram
                {X}{\catDD_{n-1}}{\catDD [X]_n;}
                {\tgt}{}{\tgt}
            \diagramarrows{}{}{}{}
        \]
        \item if $\catEE$ is a strict $n$-category, $f : \catDD \longrightarrow
        \catEE_{\leq n-1}$ is an $(n-1)$-functor, and $f_n : X \longrightarrow
        \catEE_n$ is a map such that for all $x \in X$, $f (\src x) = \src f_n
        (x)$ and $f (\tgt x) = \tgt f_n (x)$, then $f$ and $f_n$ extend
        uniquely to an $n$-functor $\catDD [X] \longrightarrow \catEE$.
    \end{enumerate}
    The free extension $\catDD [X]$ always exists and is unique up to
    isomorphism, see \cite[section 1]{Harnik2008}.
\end{definition}

For the rest of this section, let $\catDD$ be a strict $(n-1)$-category,
$\catDD \xot{\src, \tgt} X$ be a cellular extension of $\catDD$, and $\catEE
\eqdef \catDD [X]$ be the $n$-category freely generated by the cellular
extension.

\begin{definition}
    [Counting function]
    \label{def:counting-function}
    Define an $n$-category $\bbNN^{(n)}$ by
    \[
        \begin{tikzcd}
            \{ 0 \}
                \ar[r, bend right, "0" below]
            &
            \{ 0 \}
                \ar[l, "{\src, \tgt}" above]
                \ar[r, bend right, "0" below]
            &
            \phantom{\{} \cdots \phantom{\}}
                \ar[l, "{\src, \tgt}" above]
                \ar[r, bend right, "0" below]
            &
            \{ 0 \}
                \ar[l, "{\src, \tgt}" above]
                \ar[r, bend right, "0" below]
            &
            \bbNN,
                \ar[l, "{\src, \tgt}" above]
        \end{tikzcd}
    \]
    where all compositions correspond to the addition of integers. For $x \in
    X$, define a \emph{counting function} $\hash_x : X \longrightarrow \bbNN$
    that maps $x$ to $1$, and all other elements of $X$ to $0$. This extends to
    a $n$-functor $\catEE \longrightarrow \bbNN^{(n)}$. Similarly, let $\hash :
    X \longrightarrow \bbNN$ be the map sending all elements to $1$, and extend
    it as $\hash : \catEE \longrightarrow \bbNN^{(n)}$.
\end{definition}

\begin{definition}
    [{Context \cite[definition 2.1.1]{Guiraud2009}}]
    \label{def:polygraph-context}
    Consider another cellular extension
    \[
        \catDD \xot{\src, \tgt} (X + \left\{ \ctxplace \right\})
    \]
    of $\catDD$, where $\src \ctxplace$ and $\tgt \ctxplace$ are chosen
    arbitrarily. A \emph{$n$-context} of $\catEE$ is a cell $C
    \in \catDD [X + \left\{\ctxplace \right\}]_n$ such that $\hash_\ctxplace C
    = 1$. One may think of $C$ as a cell of $\catEE_n$ with a ``hole'', and we
    sometime write $C = C[\ctxplace]$. If $u \in \catEE_n$ is parallel to
    $\ctxplace$ in $\catDD [X + \left\{ \ctxplace \right\}]$, let $C [u]$, be
    $C [\ctxplace]$ where $\ctxplace$ has been replaced by $u$.
\end{definition}

\begin{definition}
    [{Category of contexts \cite[definition 2.1.2]{Guiraud2009}}]
    \label{def:category-of-contexts}
    The category $\Ctx_n \catEE$ of
    \emph{$n$-contexts} of $\catEE$ has objects the $n$-cells of $\catEE$, and
    a morphism $C : x \longrightarrow y$ is an $n$-context $C = C [\ctxplace]$
    such that $C [x] = y$. If $D : y \longrightarrow z$ is another context,
    then the composite of $C$ and $D$ is $DC \eqdef D [C [\ctxplace]] : x
    \longrightarrow z$, as indeed, $D [C [x]] = D [y] = z$.
\end{definition}

\begin{definition}
    [Primitive context]
    \label{def:primitive-context}
    A context is \emph{primitive} over a cell $y \in
    \catEE_n$ if it is of the form $C : x \longrightarrow y$ with $x \in X$.
\end{definition}

\subsection{Many-to-one polygraphs}

\begin{definition}
    [{Polygraph \cite[definition 7.1]{Harnik2008}}]
    \label{def:polygraph}
    A \emph{polygraph} (also called a
    \emph{computad}) $\catPP$ consists of a
    small $\omega$-category $\catCC$ and sets $\catPP_n \subseteq \catCC_n$ for
    all $n \in \bbNN$, such that $\catPP_0 = \catCC_0$, and such that
    $\catCC_{\leq n+1} = \catCC_{\leq n} [\catPP_{n+1}]$, i.e. the underlying
    $(n+1)$-category of $\catCC$ is freely generated by $\catPP_{n+1}$ over its
    underlying $n$-category. We usually write $\catPP^*$ instead of $\catCC$. A
    polygraph $\catPP$ is an \emph{$n$-polygraph}
    if $\catPP_k =
    \emptyset$ whenever $k > n$. A \emph{morphism of polygraphs}
    is an $\omega$-functor mapping generators to
    generators. Let $\Pol$ be the category
    of polygraphs and morphisms between them.
\end{definition}

\begin{example}
    A $1$-polygraph $\catPP$ is simply a free $1$-category generated by the
    graph $\catPP_0 \xot{\src, \tgt} \catPP_1$.
\end{example}

\begin{proposition}
    \label{prop:polygraphs-cocomplete}
    The category $\Pol$ is cocomplete. If $F : \calJJ \longrightarrow \Pol$ is
    diagram, and $n \in \bbNN$, then $(\colim_{i \in \calJJ} Fi)_n \cong
    \colim_{i \in \calJJ} (Fi)_n$.
\end{proposition}

\begin{notation}
    If $\catPP \in \Pol$, we write $\Ctx_n \catPP$ instead of $\Ctx_n
    \catPP_{\leq n-1} [\catPP_n]$ (see \cref{def:category-of-contexts}).
\end{notation}

\begin{proposition}
    [{\cite[proposition 2.1.3]{Guiraud2009}}]
    \label{prop:Guiraud2009:2.1.3}
    Let $\catPP$ be a polygraph, and $C \in \Ctx_n \catPP$. Then $C$ decomposes
    as
    \[
        C \speq d_n \graft_{n-1} (
            d_{n-1} \graft_{n-2} \cdots \: (
                d_1 \graft_0 \ctxplace \graft_0 e_1
            ) \: \cdots \graft_{n-2} e_{n-1}
        ) \graft_{n-1} e_n ,
    \]
    where $d_n, e_n \in \catPP^*_n$, and for $1 \leq i < n$, $d_i$ and $e_i$
    are identities of $i$-cells.
\end{proposition}

\begin{definition}
    [{Whisker \cite[paragraph 2.1.4]{Guiraud2009}}]
    \label{def:whisker}
    Let $\catPP$ be a polygraph. an $n$-\emph{whisker} of
    $\catPP$ is an $n$-context of the form
    $
        d_{n-1} \graft_{n-2} \cdots \: (
            d_1 \graft_0 \ctxplace \graft_0 e_1
        ) \: \cdots \graft_{n-2} e_{n-1}
    $,
    where for $1 \leq i \leq n-1$, $d_i$ and $e_i$ are identities of $i$-cells.
\end{definition}

\begin{remark}
    \label{rem:context-to-whisker}
    If $C$ is an $(n-1)$-context, then by \cref{prop:Guiraud2009:2.1.3}, it
    decomposes as on the left, and induces an $n$-whisker on the right
    \[
        C [\ctxplace] \speq d_{n-1} \graft_{n-2} \cdots \: (
            d_1 \graft_0 \ctxplace \graft_0 e_1
        ) \: \cdots \graft_{n-2} e_{n-1} ,
        \qqquad
        \id_{d_{n-1}} \graft_{n-2} \cdots \: (
            \id_{d_1} \graft_0 \ctxplace \graft_0 \id_{e_1}
        ) \: \cdots \graft_{n-2} \id_{e_{n-1}} ,
    \]
    which we shall also denote by $C$.
\end{remark}

\begin{proposition}
    [{\cite[proposition 2.1.5]{Guiraud2009}}]
    \label{prop:Guiraud2009:2.1.5}
    Let $\catPP$ be a polygraph, $u \in \catPP^*_n$, $k \eqdef \hash u$, and
    assume $k \geq 1$. Then $u$ decomposes as
    \[
        u \speq C_1 [x_1] \graft_{n-1} C_2 [x_2] \graft_{n-1} \cdots
            \graft_{n-1} C_k [x_k] ,
    \]
    where all of the $C_i$'s are $n$-whiskers, and $x_1, \ldots, x_k \in
    \catPP_n$.
\end{proposition}

\begin{definition}
    [{Partial composition \cite[definition 3.8]{Harnik2008}}]
    \label{def:partial-composition}
    Let $\catPP$ be a polygraph, $x, y \in \catPP_n^*$ be $n$-cells, and $C :
    \tgt y \longrightarrow \src x$ be a context. The \emph{partial composition}
    (called \emph{placed composition} in \cite[definition 3.8]{Harnik2008}) $x
    \graft_C y$ is defined as
    $
        x \graft_C y \speqdef x \graft_{n-1} C [y] ,
    $
    where the notation $C[y]$ follows \cref{rem:context-to-whisker}.
\end{definition}

\begin{lemma}
    \label{lemma:partial-composition}
    With $x$, $y$, and $C$ as in \cref{def:partial-composition}, we have $\src
    ( x \graft_C y ) = C [\src y]$ and $\tgt ( x \graft_C y ) = \tgt x$.
\end{lemma}
\begin{proof}
    We have $\tgt (x \graft_C y) = \tgt (x \graft_{n-1} C [y]) = \tgt x$. On
    the other hand, $\src (x \graft_C y) = \src C [y]$. By
    \cref{prop:Guiraud2009:2.1.3,rem:context-to-whisker}, $C [y]$ decomposes as
    \[
        C [y] \speq
        \id_{d_{n-1}} \graft_{n-2} \cdots \: (
            \id_{d_1} \graft_0 y \graft_0 \id_{e_1}
        ) \: \cdots \graft_{n-2} \id_{e_{n-1}} ,
    \]
    where for $1 \leq k \leq n$, $d_k$ and $e_k$ are identities of $k$-cells.
    Thus,
    \[
        \src C [y]
        \speq d_{n-1} \graft_{n-2} \cdots \: (
            d_1 \graft_0 (\src y) \graft_0 e_1
        ) \: \cdots \graft_{n-2} e_{n-1}
        \speq C [\src y] .
    \]
\end{proof}

\begin{definition}
    [{Many-to-one polygraph \cite[definition 7.4]{Harnik2008}}]
    \label{def:many-to-one-polygraph}
    Let $\catPP \in \Pol$ be a polygraph. For $n \geq 1$, an $n$-cell $x \in
    \catPP^*_n$ is said \emph{many-to-one} if $\tgt x \in \catPP_{n-1}$, and we
    write $\catPP^\MTO_n$ for the set of many-to-one $n$-cells of $\catPP$. By
    convention, all $0$-cells are many-to-one. In turn, the polygraph $\catPP$
    is called \emph{many-to-one} (or \emph{opetopic}) if all its generators are
    many-to-one. Let $\PolMTO$ be the full subcategory of $\Pol$ spanned by
    many-to-one polygraphs.
\end{definition}

\begin{lemma}
    \label{lemma:whisker-mto}
    Let $u \in \catPP^*_n$ be such that $\hash u \geq 1$. By
    \cref{prop:Guiraud2009:2.1.5}, it decomposes as
    \[
        u \speq C_1 [x_1] \graft_{n-1} C_2 [x_2] \graft_{n-1} \cdots
            \graft_{n-1} C_k [x_k] ,
    \]
    where $k \eqdef \hash u$, where all of the $C_i$'s are $n$-whiskers, and
    $x_1, \ldots, x_k \in \catPP_n$. Then $u$ is a many-to-one cell if and only
    if $C_1 = \ctxplace$, i.e. if $C_1$ is the trivial context.
\end{lemma}
\begin{proof}
    First, note that $\tgt u = \tgt C_1 [x_1]$. Write $C_1$ as
    \[
        C_1 [\ctxplace] \speq d_{n-1} \graft_{n-2} \cdots \: (
            d_1 \graft_0 \ctxplace \graft_0 e_1
        ) \: \cdots \graft_{n-2} e_{n-1} ,
    \]
    where for $1 \leq i \leq n-1$, $d_i$ and $e_i$ are identities of $i$-cells
    (see \cref{def:whisker}). We have
    \[
        \tgt u
        \speq \tgt C_1 [x_1]
        \speq (\tgt d_{n-1}) \graft_{n-2} \cdots \: (
            (\tgt d_1) \graft_0 (\tgt x_1) \graft_0 (\tgt e_1)
        ) \: \cdots \graft_{n-2} (\tgt e_{n-1}) .
    \]
    Thus, $\tgt u$ is a generator if and only if $\tgt d_i$ and $\tgt e_i$ are
    identities, for all $0 \leq i \leq n-1$. In this case, $d_i$ and $e_i$ are
    identity cells of $(i-1)$-cells, thus $C_1 = \ctxplace$. Conversely, if
    $C_1 = \ctxplace$, then $\tgt u = \tgt x_1$ is a generator since $\catPP$
    is a many-to-one polygraph, thus $u$ is a many-to-one cell.
\end{proof}

The following result comes as a polygraphic analogue to
\cref{prop:polynomial-functor:trees-are-graftings}.

\begin{proposition}
    \label{prop:mto-induction}
    A many-to-one $n$-cell $u$ of $\catPP$ is of either of the following forms:
    \begin{enumerate}
        \item $\id_a$ for $a \in \catPP_{n-1}$,
        \item $x \in \catPP_n$,
        \item $v \graft_C x = v \graft_{n-1} C [x]$, for some $v \in
        \catPP^\MTO_n$ with $1 \leq \hash v < \hash u$, $x \in \catPP_n$, and
        $C : \tgt x \longrightarrow \src v$.
    \end{enumerate}
\end{proposition}
\begin{proof}
    If $\hash u = 0$, then $u = \id_a$ for some $a
    \in P^*_{n-1}$. Further, $a = \tgt u \in \catPP_n$. If $\hash u = 1$, then
    by \cref{lemma:whisker-mto}, $u$ is necessarily a generator. If $\hash u =
    k \geq 2$, then by \cref{prop:Guiraud2009:2.1.5}, $u$ decomposes as
    \[
        u \speq C_1 [x_1] \graft_{n-1} C_2 [x_2] \graft_{n-1} \cdots
            \graft_{n-1} C_k [x_k] ,
    \]
    where all of the $C_i$'s are $n$-whiskers, and $x_1, \ldots, x_k \in
    \catPP_n$. Let
    \[
        v \speqdef
        C_1 [x_1] \graft_{n-1} C_2 [x_2] \graft_{n-1} \cdots
            \graft_{n-1} C_{k-1} [x_{k-1}] .
    \]
    Then $C_k$ is a context $\tgt x_k \longrightarrow \src v$, and $u = v
    \graft_{C_k} x_k$. By \cref{lemma:whisker-mto}, and since $u$ is
    many-to-one, $C_1 = \ctxplace$. By \cref{lemma:whisker-mto} again, $v$ is
    many-to-one, and $\hash v = k-1 \geq 1$, finishing the proof.
\end{proof}

\begin{notation}
    For $\catPP \in \PolMTO$, let $\Ctx_n^\MTO \catPP$ be the full subcategory
    of $\Ctx_n \catPP$ generated by many-to-one $n$-cells. In other words, an
    $n$-context $C : u \longrightarrow v$ is in $\Ctx_n^\MTO \catPP$ if $u, v
    \in \catPP^\MTO_n$. Necessarily, such a context is itself a many-to-one
    cell, as $\tgt C [\ctxplace] = \tgt C [u] = \tgt v$ is a generator.
\end{notation}

\begin{definition}
    \label{def:terminal-mtop}
    Define a polygraph $\catTT \in \PolMTO$by $\catTT_0
    \eqdef \{ \optZero \}$, $\catTT_{n+1} \eqdef \left\{ (u, v) \in
    \catTT_n^\MTO \times \catTT_n \mid u \parallel v \right\}$ (see
    \cref{def:many-to-one-polygraph,def:parallel-cells-hcat}), with $\src (u,
    v) \eqdef u$ and $\tgt (u, v) \eqdef v$.
\end{definition}

\begin{proposition}
    \label{prop:terminal-mtop}
    The polygraph $\catTT$ is terminal in $\PolMTO$.
\end{proposition}
\begin{proof}
    For $\catPP \in \PolMTO$, we show that there exists a unique morphism
    $f : \catPP \longrightarrow \catTT$.
    \begin{itemize}
        \item (Existence) If $x \in \catPP_0$, let $f (x) \eqdef \optZero$,
        and if $x \in \catPP_n$ with $n \geq 1$, let $f (x) = (f (\src x), f
        (\tgt x))$. The source and target compatibility is trivial.
        \item (Uniqueness) Consider $g : \catPP \longrightarrow \catTT$.
        Necessarily $g_0 = f_0$ as $\catTT_0$ is a singleton. Let $x \in
        \catPP_n$ with $n \geq 1$. We have
        \begin{equation*}
            f_n (x)
            \speq (f_{n-1} (\src x), f_{n-1} (\tgt x))
            \speq (g_{n-1} (\src x), g_{n-1} (\tgt x))
            \speq (\src g_{n-1} (x), \tgt g_{n-1} (x))
            \speq g_n (x)
        \end{equation*}
        Therefore, $f = g$.
        \qedhere
    \end{itemize}
\end{proof}

\begin{notation}
    \label{not:terminal-map}
    If $\catPP$ is a many-to-one polygraph, we write $\shriek : \catPP
    \longrightarrow \catTT$ for the terminal map.
\end{notation}

\begin{definition}
    \label{def:effective-and-familially-representable}
    \begin{enumerate}
        \item An \emph{effective category} is a
        category $\catCC$ equipped with a functor $F : \catCC \longrightarrow
        \Set$. For example:
        \begin{enumerate}
            \item if $\catAA$ is a small category, then $\Psh\catAA$ (or any
            subcategory thereof) is naturally an effective category with the
            functor $\Psh\catAA \longrightarrow \Set$ mapping a presheaf $X$ to
            $\sum_{a \in \catAA} X_a$;
            \item $\Pol$ (or any subcategory thereof) is an effective category,
            where the functor $\Pol \longrightarrow \Set$ maps a polygraph
            $\catPP$ to $\sum_{n \in \bbNN} \catPP_n$.
        \end{enumerate}

        \item A category $\catCC$ is an \emph{effective presheaf category} if
        it is effective, and equivalent, as an effective category\footnote{i.e.
        the equivalence functor commutes with the equiped functors to $\Set$.},
        to a presheaf category.

        \item A functor $F : \catCC \longrightarrow \Set$ is \emph{familially
        representable} \cite[definition 2.4]{Carboni1995} if $F \spcong \sum_{i
        \in I} \catCC (c_i, -)$ for some family $\left\{ c_i \mid i \in I
        \right\}$ of objects of $\catCC$.
    \end{enumerate}
\end{definition}

\begin{theorem}
    \label{th:Henry2019}
    \begin{enumerate}
        \item \cite[corollary 2.4.9]{Henry2019} The category $\PolMTO$ is a
        \emph{good class of polygraphs} \cite[definition 2.2.2]{Henry2019}, and
        in particular
        \begin{enumerate}
            \item it is an effective presheaf category;
            \item for all $n \in \bbNN$, the functor $(-)^*_n : \PolMTO
            \longrightarrow \Set$ that maps a polygraph $\catPP \in \PolMTO$ to
            its set $\catPP_n^*$ of $n$-cells is familially representable, and
            the representing objects are called the \emph{opetopic
            $n$-polyplexes} (or just \emph{$n$-polyplexes}):
            \[
                \catPP_n^* \spcong \sum_{
                    \substack{\omega \text{ is an opetopic} \\
                         n \text{-polyplex}}
                } \PolMTO (\omega, \catPP) .
            \]
        \end{enumerate}

        \item \cite[proposition 2.2.6]{Henry2019} The opetopic $n$-polyplexes
        are in bijective correspondence with $\catTT^*_n$. Isomorphic
        polyplexes are equal. If $u \in \catTT^*_n$, let $\uU$ be the
        associated polyplex (refer to \cite[section 2.3]{Henry2019} for the
        precise construction). The isomorphism above can be reformulated as
        \[
            \triangleURdiagram
                {\catPP_n^*}{\sum_{u \in \catTT^*_n} \PolMTO (\uU, \catPP)}
                    {\catTT_n^* ,}
                {\cong}{\shriek}{}
        \]
        where the vertical morphism maps an element in the $u \in \catTT^*_n$
        component of the sum to $u$. In other words, a cell $v \in \catPP^*_n$
        corresponds to a unique morphism of the form $\uU \longrightarrow
        \catPP$, and $u = \shriek v$ (see \cref{not:terminal-map}).

        \item \cite[lemma 2.4.4, corollary 2.3.13]{Henry2019} Let $0 \leq k <
        n$, $a \in \catTT^*_k$, and $u, v \in \catTT^*_n$ be such that
        $\tgt^{n-k} v = a = \src^{n-k} u$. Then we have natural maps
        $\src^{n-k} : \uA \longrightarrow \uU$ and $\tgt^{n-k} : \uA
        \longrightarrow \uV$, and $\underline{u \graft_k v}$ is obtained as the
        pushout
        \[
            \pushoutdiagram
                {\uA}{\uV}{\uU}{\underline{u \graft_k v} .}
                {\tgt^{n-k}}{\src^{n-k}}{\iota_v}{\iota_u}
        \]
        Furthermore, the maps $\iota_u$ and $\iota_v$ are injective on $(n-1)$-
        and $n$-cells.

    \end{enumerate}
\end{theorem}

\begin{example}
    \label{ex:polyplex}
    By \cref{def:terminal-mtop}, $\catTT$ has a unique $0$-cell $\optZero$, and
    the corresponding polyplex $\underline{\optZero}$ is simply the polygraph
    with a single $0$-generator. Indeed, $\PolMTO (\underline{\optZero}, -)$
    maps a polygraph $\catPP \in \PolMTO$ to its set of $0$-cells $\catPP_0$.

    If we write $\optOne \eqdef (\optZero, \optZero)$ for the unique
    $1$-generator of $\catTT$, then
    \[
        \catTT^*_1 \speq \left\{
            \:\:
            \id_{\optZero}, \:\:
            \optOne, \:\:
            \optOne \graft_0 \optOne, \:\:
            \optOne \graft_0 \optOne \graft_0 \optOne, \:\:
            \optOne \graft_0 \optOne \graft_0 \optOne \graft_0 \optOne, \:\:
            \ldots \:\:
        \right\} .
    \]
    Write $l_0 \eqdef \id_{\optZero}$, and $l_k$ for the composite $\optOne
    \graft_0 \cdots \graft_0 \optOne$ of $k$ instances of $\optOne$. Then the
    polyplex $\underline{l_0}$ is simply $\underline{\optZero}$, and
    $\underline{l_k}$ spans the free category on the linear graph with $k$
    vertices
    $
        \bullet \longrightarrow
        \bullet \longrightarrow
        \cdots \longrightarrow \bullet
    $.
    Indeed, let $\catPP \in \PolMTO$. Then a $1$-cell $u$ of $\catPP$ is either
    \begin{enumerate}
        \item an identity of a $0$-cell;
        \item a sequence of $k$ composable $1$-generators of $\catPP$.
    \end{enumerate}
    If $u = \id_a$ for some $a \in \catPP_0$, then it is uniquely identified
    (as a $1$-cell) by a morphism $\underline{l_0} \longrightarrow \catPP$
    mapping the unique $0$-cell of $\underline{l_0}$ to $a$. If $u$ is a
    composite of $k$ generators, then uniquely identified (as a $1$-cell) by a
    morphism $\underline{l_k} \longrightarrow \catPP$. In conclusion,
    \[
        \catPP^*_1
        \spcong \sum_{k \in \bbNN} \PolMTO (\underline{l_k}, \catPP)
        \speq \sum_{v \in \catTT^*_1} \PolMTO (\uV, \catPP) .
    \]
\end{example}

\begin{remark}
    If $\uU$ is an $n$-polyplex, then under the isomorphism of
    \cref{th:Henry2019} (2), the identity morphism $\uU \longrightarrow \uU$
    corresponds to an $n$-cell of $\uU$, which we call its \emph{fundamental
    cell}, and following \cite{Henry2019}, denote by $u$. If $\catPP$ is a
    many-to-one polygraph, $v \in \catPP^*_n$, and $u = \shriek v$, then the
    map $f : \uU \longrightarrow \catPP$ corresponding to $v$ maps the
    fundamental cell $u$ to $v$. Indeed, by naturality of the isomorphism of
    \cref{th:Henry2019} (2), the diagram on the left commutes, which results in
    the mappings displayed on the right:
    \[
        \squarediagram
            {\uU^*_n}{\sum_{w \in \catTT^*_n} \PolMTO (\uW, \uU)}
                {\catPP^*_n}{\sum_{w \in \catTT^*_n} \PolMTO (\uW, \catPP)}
            {\cong}{f}{f \circ -}{\cong}
        \qqquad
        \begin{tikzcd}
            u
                \ar[r, mapsto]
                &
            \id_{\uU}
                \ar[d, mapsto]
                \\
            v
                \ar[r, mapsto]
                &
            f.
        \end{tikzcd}
    \]
    Necessarily, $f (u) = v$.
\end{remark}

\begin{lemma}
    [{\cite[lemma 2.4.5]{Henry2019}}]
    \label{lemma:Henry2019:2.4.5}
    Let $n \geq 1$, $\uU$ be an $n$-polyplex, and $u$ be its fundamental cell.
    For $a \in \uU_{n-1}$, exactly one of the following two possibilities
    is true:
    \begin{enumerate}
        \item $a$ occurs in $\src u$;
        \item $a$ is the target of a $n$-generator of $\uU$.
    \end{enumerate}
    Furthermore, in the second case, the $n$-generator in question is unique.
\end{lemma}

\begin{lemma}
    [{\cite[remark 2.2.9, corollary 2.2.13]{Henry2019}}]
    \label{lemma:cartesian-compositions}
    Let $f : \catPP \longrightarrow \catQQ$ a morphism of many-to-one
    polygraphs, and recall from \cref{def:omega-category} that $\catPP^*_{n, k}
    = \catPP^*_n \times_{\catPP^*_k} \catPP^*_n$. The following square is
    cartesian
    \[
        \squarediagram
            {\catPP^*_{n, k}}{\catPP^*_n}{\catQQ^*_{n, k}}{\catQQ^*_n .}
            {\graft_k}{f}{f}{\graft_k}
    \]
    Consequently, for $u_1, u_2, v_1, v_2 \in \catPP^*_n$, if $u_1 \graft_k v_1
    = u_2 \graft_k v_2$, $f (u_1) = f(u_2)$, and $f (v_1) = f(v_2)$, then $u_1
    = u_2$ and $v_1 = v_2$.
\end{lemma}
\begin{proof} [Proof (sketch)]
    Let us first consider the case $\catQQ = \catTT$, and let $u, v \in
    \catPP_n^*$ be such that $\tgt^{n-k} v = \src^{n-k} u = a$. In particular,
    pair $(u, v)$ is in $\catPP^*_{n, k}$. We have a series of correspondences
    \begin{prooftree}
        \AxiomC{A tuple $(u, v) \in \catPP^*_{n, k}$}
        \RightLabel{\cref{th:Henry2019} (2)}
        \UnaryInfC{A map $
            \underline{\shriek u} \coprod_{\underline{\shriek a}}
                \underline{\shriek v}
            \longrightarrow \catPP
        $}
        \RightLabel{\cref{th:Henry2019} (3)}
        \UnaryInfC{A map $
            \underline{\shriek (u \graft_k v)} \longrightarrow \catPP
        $ with a decomposition of $\shriek (u \graft_k v)$ as $x \graft_k y$}
        \RightLabel{\cref{th:Henry2019} (2)}
        \UnaryInfC{An element $u \graft_k v \in \catPP^*_n$ with a
            decomposition of $\shriek (u \graft_k v)$ as $x \graft_k y$.}
    \end{prooftree}
    In other words, the following square is a pullback
    \[
        \squarediagram
            {\catPP^*_{n, k}}{\catPP^*_n}{\catTT^*_{n, k}}{\catTT^*_n .}
            {\graft_k}{\shriek}{\shriek}{\graft_k}
    \]
    For the general case, note that in the following diagram, the lower and
    outer squares are cartesian, and by the pasting lemma, so is the upper one:
    \[
        \begin{tikzcd}
            \catPP^*_{n, k}
                \ar[r, "\graft_k"] \ar[d, "f" right]
                \ar[dd, bend right, "\shriek" left] &
            \catPP^*_n
                \ar[d, "f" left] \ar[dd, bend left, "\shriek" right] \\
            \catQQ^*_{n, k}
                \pullbackcorner
                \ar[r, "\graft_k"] \ar[d, "\shriek" right] &
            \catQQ^*_n
                \ar[d, "\shriek" left] \\
            \catTT^*_{n, k}
                \ar[r, "\graft_k"] &
            \catTT^*_n .
        \end{tikzcd}
    \]
\end{proof}

Let $\catPP \in \PolMTO$ and $v \in \catPP_n^\MTO$. Then $v$ is a composition
of $n$-generators of $\catPP$, which are many-to-one, so intuitively, $v$ is a
``tree of $n$-generators''. In this section, we make this idea formal. We first
define a polynomial functor $\nabla_n \catPP$ whose operations are the
$n$-generators of $\catPP$ (\cref{def:nabla-construction}), and then construct
the \emph{composition} $\composition{T} \in \catPP_n^\MTO$ of a $\nabla_n
\catPP$-tree $T$. In \cref{proposition:composition-map-bijective}, we show that
this construction is bijective.

\begin{definition}
    [The $\nabla$ construction]
    \label{def:nabla-construction}
    For $\catPP \in \PolMTO$ and $n \geq 1$, let $\nabla_n \catPP$
    be the following polynomial endofunctor:
    \[
        \polynomialfunctor
            {\catPP_{n-1}}{\catPP^\nodesymbol_n}{\catPP_n}{\catPP_{n-1},}
            {\src}{p}{\tgt}
    \]
    where for $x \in \catPP_n$, the fiber $\catPP^\nodesymbol_n (x)$ is the set
    of primitive contexts over $\src x$, and for $C : a \longrightarrow \src x$
    in $\catPP^\nodesymbol_n (x)$, $\src (C) \eqdef a$, $p (C) \eqdef x$, and
    $\tgt$ is the target map of $\catPP$.
\end{definition}

\begin{lemma}
    \label{lemma:primitive-context-of-mto-cells}
    Let $\catPP, \catQQ \in \PolMTO$, and $f : \catPP \longrightarrow \catQQ$.
    For $v \in \catPP_n^\MTO$, write $E (v)$ for the set of primitive contexts
    over $v$, and likewise for many-to-one cells of $\catQQ$. Then $f$ induces
    a bijection $E(v) \longrightarrow E(f (v))$.
\end{lemma}
\begin{proof}
    We proceed by induction on $v$ (see \cref{prop:mto-induction}).
    \begin{enumerate}
        \item If $v$ is an identity (resp. a generator), then so is $f (v)$,
        thus $E (v)$ and $E (f (v))$ are both empty (resp. singletons).
        Trivially, $f : E (v) \longrightarrow E (f (v))$ is a bijection.
        \item Assume that $v$ decomposes as $v = w \graft_C x$ with $\hash w
        \geq 1$ and $x \in \catPP_n$, and $C : \tgt x \longrightarrow \src w$.
        Then a primitive context over $v$ is either $w \graft_C \ctxplace$ or
        of the form $D [\ctxplace] \graft_C x$ for $D \in E(w)$, and $E (v)
        \cong 1 + E (w)$. Likewise, $E (f (v)) \cong 1 + E (f (w))$, and it is
        straightforward to check that $f : E (v) \longrightarrow E (f (v))$ is
        indeed a bijection.
        \qedhere
    \end{enumerate}
\end{proof}

\begin{proposition}
    \label{prop:nabla-construction-induced-morphism}
    Let $f : \catPP \longrightarrow \catQQ$ be a morphism of many-to-one
    polygraphs. For all $n \geq 1$, it induces a morphism of polynomial
    functors $\nabla_n f : \nabla_n \catPP \longrightarrow \nabla_n \catQQ$,
    where $(\nabla_n f)_1 = f_n : \catPP_n \longrightarrow \catQQ_n$.
\end{proposition}
\begin{proof}
    Consider
    \[
        \begin{tikzcd}
            \catPP_{n-1}
                \ar[d, "f"'] &
            \catPP^\nodesymbol_n
                \ar[r, "p"]
                \ar[d, "f^\nodesymbol" left]
                \ar[l, "\src" above] &
            \catPP_n
                \ar[r, "\tgt"]
                \ar[d, "f" left] &
            \catPP_{n-1}
                \ar[d, "f" left] \\
            \catQQ_{n-1} &
            \catQQ^\nodesymbol_n
                \ar[r, "p"]
                \ar[l, "\src" above] &
            \catQQ_n
                \ar[r, "\tgt"] &
            \catQQ_{n-1}
        \end{tikzcd}
    \]
    where $f^\nodesymbol_n$ maps a context $C : a \longrightarrow \src x$ to $f
    (C) : f (a) \longrightarrow f (\src x)$. Clearly, all squares commute, and
    by \cref{lemma:primitive-context-of-mto-cells}, the middle one is
    cartesian.
\end{proof}

\begin{definition}
    [Composition]
    \label{def:composition}
    We define the \emph{composition} operation
    $\composition{(-)} : \tree \nabla_n \catPP \longrightarrow
    \catPP_n^\MTO$. At
    the same time, we establish a bijection between $T^\leafsymbol$ and the
    primitive contexts over $\src T^\circ$, where $T \in \tree \nabla_n
    \catPP$.
    \begin{enumerate}
        \item If $a \in \catPP_{n-1}$, then $(\itree{a})^\circ \eqdef \id_a$.
        Note that the only primitive context over $\src \id_a$ is $\ctxplace :
        a \longrightarrow a$, and let $C_{[]} \eqdef \ctxplace$.

        \item If $x \in \catPP_n$, then $(\ytree{x})^\circ \eqdef x$. Note that
        by definition of $\nabla_n \catPP$ (\cref{def:nabla-construction}) we
        have $\ytree{x}^\leafsymbol = \left\{ [D] \mid D \in x^\nodesymbol
        \right\}$ (see \cref{rem:elementary-p-trees:addresses}), and let
        $C_{[D]} \eqdef D$.

        \item Consider a tree of the form $S = T \graft_{[l]} \ytree{x}$, with
        $T \in \tree \nabla_n \catPP$ having at least one node, $[l] \in
        T^\leafsymbol$ and $x \in \catPP_n$. By induction, the leaf $[l]$
        corresponds to a primitive context $C_{[l]} : a \longrightarrow \src
        T^\circ$, and moreover, $a = \tgt x$. Let $S^\circ \eqdef T^\circ
        \graft_{C_{[l]}} x$. Let $[l'] \in S^\leafsymbol$. If $[l']$ is of the
        form $[l D]$, for some $[D] \in \ytree{x}^\leafsymbol$, let $C_{[l']}
        \eqdef C_{[l]} [D]$. Otherwise, $[l']$ is a leaf of $T$, and so
        $C_{[l']}$ is already defined.

        If $S = T' \graft_{[l']} \ytree{x'}$ is another decomposition of $S$,
        then the fact that $T^\circ \graft_{C_{[l]}} x = (T')^\circ
        \graft_{C_{[l']}} x'$ follows from the exchange law.
    \end{enumerate}
\end{definition}

\begin{definition}
    [Composition tree]
    \label{def:composition-tree}
    A \emph{composition tree} is simply a $\nabla_n
    \catPP$-tree. If $v \in \catPP_n^\MTO$ and $T$ is a composition tree such
    that $\composition{T} = v$, then we say that $T$ is a composition tree of
    $v$.
\end{definition}

The following result generalizes \cref{lemma:cartesian-compositions}.

\begin{proposition}
    \label{prop:cartesian-compositions}
    Let $f : \catPP \longrightarrow \catQQ$ a morphism of many-to-one
    polygraphs. The following square is cartesian
    \[
        \squarediagram
            {\tree \nabla_n \catPP}{\catPP_n^\MTO}
                {\tree \nabla_n \catQQ}{\catQQ_n^\MTO.}
            {\composition{(-)}}{f}{f}{\composition{(-)}}
    \]
\end{proposition}
\begin{proof}
    This amounts to showing that if $v \in \catPP_n^\MTO$, then $f$ establishes
    a bijective correspondence between the composition trees of $v$ and $f
    (v)$. This is clear if $\hash v \leq 1$, so let us assume $\hash v \geq 2$,
    and let $T$ be a composition tree of $v$. Then $T$ decomposes as
    \[
        T \speq \ytree{a} \biggraft_{[C]} T_C ,
    \]
    where $a \in \catQQ_n$, $C$ ranges over the primitive contexts over $\src
    a$, and $T_C \in \tree \nabla_n \catQQ$. Then, by
    \cref{lemma:cartesian-compositions}, there exists a unique $b \in \catPP_n$
    such that $f(b) = a$, and unique $v_C$'s such that $f(v_C) =
    \composition{T_C}$. Further, $f$ exhibits a bijection between the primitive
    contexts over $a$ and over $b$. Note that $\hash v_C < \hash v$, so by
    induction, there exists a unique $S_C \in \tree \nabla_n \catPP$ such that
    $f (S_C) = T_C$. Finally,
    \[
        S \speqdef \ytree{b} \biggraft_{[D]} S_{f(D)}
    \]
    is the unique tree such that $f(S) = T$.
\end{proof}

\begin{lemma}
    \label{lemma:composition-tree:fundamental-cell}
    Let $\uU$ be a polyplex. The fundamental cell $u \in \uU_n$ has at most one
    composition tree.
\end{lemma}
\begin{proof}
    Assume that $S$ and $T$ are two composition trees of $u$. Necessarily,
    $\edg_{[]} S = \tgt u = \edg_{[]} T$, i.e. the root edges of $S$ and $T$
    are both decorated by the target of $u$ (recall the notation from
    \cref{not:edg,def:o}). By \cref{lemma:Henry2019:2.4.5}, exactly one of the
    following two possibilities is true.
    \begin{enumerate}
        \item If $\tgt u$ occurs in $\src u$, then $u$ is an identity cell, and
        $S = \itree{\tgt u} = T$. We are done.

        \item Otherwise $\tgt u$ is the target of a unique $n$-generator of
        $\uU$, say $x$. In particular, $\src_{[]} S = x = \src_{[]} T$, i.e.
        the root nodes of $S$ and $T$ are decorated by $x$. Let $a$ be an
        $(n-1)$-generator occurring in $\src x$. It decorates an input edge $e$
        (resp. $e'$) of the root node of $S$ (resp. $T$). If $a$ occurs in
        $\src u$, then by \cref{lemma:Henry2019:2.4.5} again, there is no
        $n$-generator whose target is $a$. So in $S$ (resp. $T$), there cannot
        be a node above $e$ (resp. $e'$), i.e. $e$ (resp. $e'$) is a leaf.
        Otherwise, $a$ is the target of a unique $n$-generator $y$, and
        necessarily, the nodes above $e$ and $e'$, namely $\src_{[e]} S$ and
        $\src_{[e']} T$ are decorated by $y$. Applying
        \cref{lemma:Henry2019:2.4.5} repeatedly, we show that $S = T$.
        \qedhere
    \end{enumerate}
\end{proof}

\begin{proposition}
    \label{proposition:composition-map-bijective}
    The composition map $\composition{(-)} : \tree \nabla_n \catPP
    \longrightarrow \catPP_n^\MTO$ is a bijection.
\end{proposition}
\begin{proof}
    \begin{itemize}
        \item (Surjectivit) This clearly holds if $n \leq 1$. Assume $n \geq
        2$ and let $v \in \catPP^\MTO_n$. If $v = \id_a$ for some $a \in
        \catPP_{n-1}$, then $v = \composition{\itree{a}}$. If $v \in \catPP_n$,
        then $v = \composition{\ytree{v}}$. Otherwise, by
        \cref{prop:mto-induction}, $v$ decomposes as $w \graft_C x$, where $w
        \in \catPP_n^\MTO$, $x \in \catPP_n$, and $C : \tgt x \longrightarrow
        \src w$ is a context. By induction, there exist $T \in \tree \nabla_n
        \catPP$ such that $\composition{T} = w$. By construction, $C$
        corresponds to a unique leaf address $[l] \in T^\leafsymbol$. Finally,
        $v = \composition{\left( T \graft_{[l]} \ytree{x} \right)}$.

        \item (Injectivit) Let $v \in \catPP_n^\MTO$, $u \eqdef \shriek v$,
        and $f : \uU \longrightarrow \catPP$ be the map associated with $v$
        (see \cref{th:Henry2019} (2)). By \cref{prop:cartesian-compositions},
        the following square is cartesian
        \[
            \pullbackdiagram
                {\tree \nabla_n \uU}{\uU_n^\MTO}
                    {\tree \nabla_n \catPP}{\catPP_n^\MTO ,}
                {\composition{(-)}}{f}{f}{\composition{(-)}}
        \]
        and in particular, since $f$ maps the fundamental cell $u$ to $v$, it
        induces a bijection between the composition trees of $u \in \uU_n^\MTO$
        and $v \in \catPP_n^\MTO$. By
        \cref{lemma:composition-tree:fundamental-cell} and surjectivity proved
        above, $u$ has exactly one composition tree, and thus, so does $v$.
        \qedhere
    \end{itemize}
\end{proof}

\begin{notation}
    Let the \emph{composition tree} operation $\ct : \catPP_n^\MTO
    \longrightarrow \tree \nabla_n \catPP$ be the inverse of the composition
    operation $\composition{(-)}$ of \cref{def:composition}.
\end{notation}

\begin{corollary}
    \label{coroll:mto-decomposition}
    Let $\catPP \in \PolMTO$ and $v \in \catPP^\MTO_n$ with $\hash v \geq 1$.
    Then $v$ uniquely decomposes as
    \[
        v \speq x \biggraft_C v_C
    \]
    where $x \in \catPP_n$, $C$ ranges over the primitive contexts over $\src
    x$, and $v_C \in \catPP^\MTO_n$.
\end{corollary}
\begin{proof}
    The composition tree of $v$ decomposes uniquely as
    \[
        \ct v \speq \ytree{x} \biggraft_{[C]} T_C ,
    \]
    and applying back $\composition{(-)}$ gives the desired decomposition of
    $v$.
\end{proof}

\begin{corollary}
    \label{coroll:mto-decomposition-context}
    Let $\catPP \in \PolMTO$. A many-to-one context $C \in \Ctx^\MTO_n \catPP$
    decomposes uniquely as
    \[
        C
        \speq C [\ctxplace]
        \speq x \graft_D \ctxplace \biggraft_E v_E
    \]
    where $x, v_E \in \catPP^\MTO_n$.
\end{corollary}
\begin{proof}
    We proceed by induction on $\hash C$. Since $\hash_\ctxplace C = 1$ (i.e.
    $C$ has only one occurrence of $\ctxplace$, see
    \cref{def:counting-function}), we have $\hash C \geq 1$. By
    \cref{coroll:mto-decomposition}, $C$ uniquely decomposes as
    \[
        C \speq y \biggraft_F u_F .
    \]
    Either one of following two possibilities is true.
    \begin{enumerate}
        \item If $y = \ctxplace$, then consider
        \[
            C \speq \id_{\tgt \ctxplace} \graft_\boxdot \ctxplace
                \biggraft_F u_F ,
        \]
        where $\boxdot$ is the trivial context $\tgt \ctxplace \longrightarrow
        \tgt \ctxplace$.

        \item Otherwise, there exists a unique $F \in y^\nodesymbol$ such that
        $\hash_\ctxplace u_F = 1$. By induction, $u_F$ decomposes as on the
        left, and consider the decomposition of $C$ as on the right:
        \[
            u_F \speq z \graft_G \ctxplace \biggraft_H w_H ,
            \qqquad
            C \speq \left( y \graft_F z \biggraft_{G \neq F} u_G \right)
                \graft_F \ctxplace \biggraft_H w_H .
        \]
        \qedhere
    \end{enumerate}
\end{proof}

\begin{remark}
    Given a many-to-one cell as on the left below,
    \cref{coroll:mto-decomposition} decomposes it as on the right.
    \[
        \tikzinput[.8]{opetope-composition}{cell}
        \qquad \longleadsto \qquad
        \tikzinput[.8]{opetope-composition}{decomposition-mto}
    \]
    Given a context as on the left below,
    \cref{coroll:mto-decomposition-context} detaches $\ctxplace$ from the cell
    ``below'' and the cells ``above'' it:
    \[
        \tikzinput[.8]{opetope-composition}{cell-ctx}
        \qquad \longleadsto \qquad
        \tikzinput[.8]{opetope-composition}{decomposition-isolation}
    \]
\end{remark}

\begin{remark}
    \label{rem:mto-contexts}
    Let $C : x \longrightarrow y$ be an $n$-context of $\catPP \in \PolMTO$
    between many-to-one cells. In particular, $C$ is a many-to-one cell in the
    extended category $\catQQ$ of \cref{def:polygraph-context}, so by virtue of
    \cref{coroll:mto-decomposition-context}, it uniquely decomposes as
    \[
        C \speq z \graft_D \ctxplace \biggraft_E v_E .
    \]
    where $z, v_E \in \catPP^*_n$. Since $C [x] = y$, we have
    \[
        y \speq z \graft_D x \biggraft_E v_E .
    \]
    Conversely, any decomposition of $y$ of the form above induces a context $C
    : x \longrightarrow y$. In particular, the number of primitive contexts
    over $y$ is $\hash y$, and $\nabla_n \catPP$ is finitary.
\end{remark}

\begin{definition}
    We now extend $\composition{(-)}$ and $\ct$ to functors between $\tree
    \nabla_n \catPP$ and $\Ctx_n^\MTO \catPP$. On objects, they are respectively
    defined in \cref{def:composition,def:composition-tree}.
    \begin{enumerate}
        \item Let $f : T \longrightarrow S$ be a morphism in $\tree \nabla_n
        \catPP$. It corresponds to a decomposition of $S$ as on the left, and
        let $f^\circ$ be the context $T^\circ \longrightarrow S^\circ$ on the
        right:
        \[
            S \speq U \graft_{[p]} T \biggraft_{[l]} V_{[l]} ,
            \qqquad
            f^\circ \speqdef U^\circ \graft_{C_{[p]}} \ctxplace
                \biggraft_{C_{[l]}} V_{[l]}^\circ ,
        \]
        where $[l]$ ranges over $T^\leafsymbol$.
        \item Let $C : x \longrightarrow y$ be an $n$-context. By
        \cref{coroll:mto-decomposition-context}, it decomposes uniquely as on
        the left, and let $\ct C$ correspond to the decomposition of $\ct y$ on
        the right:
        \[
            C [\ctxplace] \speq z \graft_{D} \ctxplace \biggraft_{E} t_E,
            \qqquad
            \ct y
            \speq (\ct z) \graft_{[l_D]} (\ct x) \biggraft_{[l_E]} (\ct t_E) ,
        \]
        where $E$ ranges over all primitive contexts over $\src x$.
    \end{enumerate}
\end{definition}

One readily checks the following:

\begin{proposition}
    [Composition tree duality]
    \label{prop:composition-tree-duality}
    The functors $(-)^\circ$ and $\ct$ are mutually inverse isomorphisms of
    categories.
\end{proposition}

\begin{corollary}
    \label{coroll:primitive-context-tree}
    For $n \geq 2$ and $x \in \catPP_n$, the functor $\ct$ induces a natural
    bijection
    \[
        \catPP^\nodesymbol_n (x)
        \spcong
        \sum_{a \in \catPP_{n-1}}
            (\tr \nabla_{n-1} \catPP) (\ytree{a}, \ct \src x) .
    \]
\end{corollary}
\begin{proof}
    Direct consequence of \cref{prop:composition-tree-duality}.
\end{proof}

\begin{notation}
    \label{not:source-of-a-generator}
    If $x \in \catPP_n$ and $[p] \in (\ct \src x)^\nodesymbol$, then we write
    $\src_{[p]} x$ instead of $\src_{[p]} \ct \src x \in \catPP_{n-1}$.
\end{notation}

\section{The equivalence}
\label{sec:equialence}

We now aim at proving that the category of opetopic sets, i.e.
$\Set$-presheaves over the category $\bbOO$ defined previously, is equivalent
to the category of many-to-one polygraphs $\PolMTO$. We achieve this by first
constructing the \emph{polygraphic realization} functor $\polyreal{-} : \bbOO
\longrightarrow \PolMTO$. This functor ``realizes'' an opetope as a polygraph
that freely implements all its tree structure by the means of adequately chosen
generators in each dimension. Secondly, we consider the left Kan extension
$\polyreal{-} : \PshO \longrightarrow \PolMTO$ along the Yoneda embedding. This
functor has a right adjoint, the ``opetopic nerve'' $N : \PolMTO
\longrightarrow \PshO$, and we prove this adjunction to be an adjoint
equivalence. This is done using the \emph{shape function}, defined in
\cref{sec:shape}, which to any generator $x$ of a many-to-one polygraph
$\catPP$ associates an opetope $\shape{x}$ along with a canonical morphism
$\tildX : \polyreal{\shape{x}} \longrightarrow \catPP$.

\subsection{Polygraphic realization}
\label{sec:polygraphic-realization}

An opetope $\omega \in \bbOO_n$, with $n \geq 1$, has one target $\tgt \omega$,
and sources $\src_{[p]} \omega$ laid out in a tree. If the sources $\src_{[p]}
\omega$ happened to be generators in some polygraph, then that tree would
describe a way to compose them. With this in mind, we define a many-to-one
$n$-polygraph $\polyreal{\omega}$, whose generators are essentially iterated
faces (i.e. sources or targets) of $\omega$ (hypothesis \condition{PR1} below).
Moreover, $\polyreal{\omega}$ will be ``maximally unfolded'' (or ``free''), in
that two (iterated) faces that are the same opetope, but located at different
addresses, will correspond to distinct generators.

The rest of this subsection is devoted to inductively define the realization
functor $\polyreal{-} : \bbOO \longrightarrow \PolMTO$ together with its
\emph{boundary} $\partial \polyreal{-}$. We bootstrap the process with
\cref{def:polygraphic-realization-low-dimension} and state our induction
hypotheses in \ref{ass:polygraphic-realization}.

\begin{definition}
    [Low dimensional cases]
    \label{def:polygraphic-realization-low-dimension}
    For $\optZero$ the unique $0$-opetope, let $\partial \polyreal{\optZero}$
    be the empty polygraph, and $\polyreal{\optZero}$ be the polygraph with a
    unique generator in dimension $0$, which we denote by $\optZero$. For
    $\optOne$ the unique $1$-opetope, let $\partial \polyreal{\optOne} \eqdef
    \polyreal{\optZero} + \polyreal{\optZero}$, and let $\polyreal{\optOne}$ be
    induced by the cellular extension
    \[
        \partial \polyreal{\optOne} \xot{\src, \tgt} \{ \optOne \} ,
    \]
    where $\src$ and $\tgt$ map $\optOne$ to distinct $0$-generators. There are
    obvious functors $\polyreal{\src_{[]}}, \polyreal{\tgt} :
    \polyreal{\optZero} \longrightarrow \polyreal{\optOne}$, mapping $\optZero$
    to $\src \optOne$ and $\tgt \optOne$, respectively.
\end{definition}

\begin{definition} [Dimension $2$]
    For the reader's convenience, we construct $\partial \polyreal{\optInt{k}}$
    and $\polyreal{\optInt{k}}$ for every opetopic integer $\optInt{k}$,
    although this case already falls under the inductive definition (see
    \cref{def:polygraphic-realization:inductive-boundary,def:polygraphic-realization:inductive}).
    \begin{enumerate}
        \item Let $\partial \polyreal{\optInt{0}}$ be the $1$-polygraph given
        by the  following coequalizer:
        \[
            \coeqdiagram
                {\polyreal{\optZero}}{\polyreal{\optOne}}{\partial \polyreal{\optInt{0}}.}
                {\src_{[]}}{\tgt}{}
        \]
        In other words, $\partial \polyreal{\optInt{0}}$ has one object $x$ and
        one generating endomorphism $f : x \longrightarrow x$. The
        $2$-polygraph $\polyreal{\optInt{0}}$ is obtained by adjoining a
        generating $2$-cell $\alpha : \id_x \longrightarrow f$ to $\partial
        \polyreal{\optInt{0}}$.
        \item Let $k \geq 1$, and consider the $1$-polygraph
        \[
            \catPP \speqdef \left(
                \polyreal{\optOne}
                \coprod_{\polyreal{\optZero}} \polyreal{\optOne}
                \coprod_{\polyreal{\optZero}} \polyreal{\optOne}
                \coprod_{\polyreal{\optZero}} \cdots
                \coprod_{\polyreal{\optZero}} \polyreal{\optOne}
            \right) ,
        \]
        where there are $k$ instances of $\polyreal{\optOne}$. In other words,
        $\catPP$ is generated by a chain of $k$ composable $1$-cells, which we
        denote by
        \[
            x_0 \xto{f_1} x_1 \xto{f_2} x_2 \xto{f_3} \cdots
            \xto{f_k} x_k .
        \]
        Alternatively, $\catPP$ is the polyplex $\underline{l_k}$ of
        \cref{ex:polyplex}. Let $\partial \polyreal{\optInt{k}}$ be the obvious
        pushout
        \[
            \diagramsize{2}{3}
            \pushoutdiagram
                {\partial \polyreal{\optOne}}{\catPP}{\polyreal{\optOne}}
                    {\partial \polyreal{\optInt{k}},}
                {(x_0, x_k)}{}{}{}
        \]
        i.e., $\partial \polyreal{\optInt{k}}$ is $\catPP$ with an additional
        generating $1$-cell $g : x_0 \longrightarrow x_k$. Finally,
        $\polyreal{\optInt{k}}$ is obtained from $\partial
        \polyreal{\optInt{k}}$ by adjoining a generating $2$-cell $\alpha : f_k
        \cdots f_2 f_1 \longrightarrow g$.
    \end{enumerate}
    In both cases, there is a bijective correspondence between the generating
    cells of $\polyreal{\optInt{k}}$ and the objects of $\bbOO / \optInt{k}$,
    and the obvious inclusions induce functors
    \[
        \partial \polyreal{-}, \polyreal{-} : \bbOO_{\leq 2}
            \longrightarrow \PolMTO .
    \]
\end{definition}

Let $n \geq 2$ and assume by induction that $\partial \polyreal{-}$ and
$\polyreal{-}$ are defined on $\bbOO_{< n}$. Assume further that the following
induction hypotheses hold (they are easily verified for $n = 2$).

\begin{assumptions}
    \label{ass:polygraphic-realization}
    For all $\psi \in \bbOO_k$ with $k < n$, the following hold:
    \begin{itemize}
        \item \condition{PR1} for all $j \in
        \bbNN$, the set $\polyreal{\psi}_j$ of $j$-generators of
        $\polyreal{\psi}$ is in bijection with the set of objects of the slice
        $\bbOO_j / \psi$, i.e. of the form $\left( \phi \xto{\sfA} \psi
        \right)$ for $\phi \in \bbOO_j$ and $\sfA : \phi \longrightarrow \psi$
        a morphism in $\bbOO$;
        \item \condition{PR2} for $\left( \phi
        \xto{\sfA} \psi \right)$ a generator of $\polyreal{\psi}$, its target
        is $\left( \tgt \phi \xto{\tgt} \phi \xto{\sfA} \psi \right)$;
        \item \condition{PR3} for $l \leq k$, and
        for $\left( \phi \xto{\sfA} \psi \right)$ a $l$-generator of
        $\polyreal{\psi}$, the composition tree of its source $\ct \src
        \left(\phi \xto{\sfA} \psi \right) \in \tree \nabla_{l - 1}
        \polyreal{\psi}$ is
        \begin{align*}
            \underlyingtree{\phi} &\longrightarrow \nabla_{l - 1} \polyreal{\psi} \\
            [p] &\longmapsto \left(
                \src_{[p]} \phi \xto{\src_{[p]}} \phi \xto{\sfA} \psi
            \right) .
        \end{align*}
        Recall that by \cref{prop:composition-tree-duality}, this completely
        determines $\src \left(\phi \xto{\sfA} \psi \right) \in
        \polyreal{\psi}_{l-1}^*$.
    \end{itemize}
\end{assumptions}

We now define $\partial \polyreal{\omega}$ and $\polyreal{\omega}$ when $\omega
\in \bbOO_n$. Defining the former is easy, and done in
\cref{def:polygraphic-realization:inductive-boundary}. The latter is defined in
\cref{def:polygraphic-realization:inductive} as generated by a cellular
extension
\[
    \partial \polyreal{\omega} \xot{\src, \tgt} \left\{ \omega \right\}
\]
of $\partial \polyreal{\omega}$, where the target and source of the new
generator are given by \condition{PR2} and \condition{PR3}. Lastly, we check
the inductive hypotheses in \cref{prop:polygraphic-realization:inductive}.

\begin{definition}
    [Inductive step for $\partial \polyreal{-}$]
    \label{def:polygraphic-realization:inductive-boundary}
    For $\omega \in \bbOO_n$, let $\partial \polyreal{\omega}$ be the following
    many-to-one $(n-1)$-polygraph:
    \[
        \partial \polyreal{\omega}
        \speqdef
        \colim_{\substack{\sfA : \psi \rightarrow \omega \\ \dim \psi < n}}
            \polyreal{\psi} .
    \]
    For $\sfA : \psi \longrightarrow \omega$ in $\bbOO_{< n} / \omega$, this
    colimit comes with a corresponding coprojection $\polyreal{\sfA} :
    \polyreal{\psi} \longhookrightarrow \partial \polyreal{\omega}$.
\end{definition}

\begin{remark}
    \label{rem:polygraphic-realization:inductive-boundary}
    Let $0 \leq k < n$. By \condition{PR1}, the set of $k$-generators of
    $\partial \polyreal{\omega}$ is $\bbOO_k / \omega$.
\end{remark}

\begin{lemma}
    \label{lemma:polygraphic-realization:generators-of-boundary}
    For $\omega \in \bbOO_n$, and $j < n$, the set $\partial
    \polyreal{\omega}_j$ of $j$-generators of $\partial \polyreal{\omega}$ is
    the slice $\bbOO_j / \omega$.
\end{lemma}
\begin{proof}
    Follows from the induction hypothesis \condition{PR1} and
    \cref{prop:polygraphs-cocomplete}.
\end{proof}

\begin{corollary}
    \label{lemma:polygraphic-realization:nabla-boundary}
    For $\omega \in \bbOO_n$ and $1 \leq k < n$, the polynomial functor
    $\nabla_k \partial \polyreal{\omega}$ is described as follows:
    \[
        \polynomialfunctor
            {\bbOO_{k-1} / \omega}{E}{\bbOO_k / \omega}
                {\bbOO_{k-1} / \omega}
            {\src}{p}{\tgt}
    \]
    where for $\left( \psi \xto{\sfA} \omega \right) \in \bbOO_k / \omega$,
    \begin{enumerate}
        \item the fiber $E \left( \psi \xto{\sfA} \omega \right)$ is simply
        $\psi^\nodesymbol$;
        \item for $[p] \in E \left( \psi \xto{\sfA} \omega \right) \cong
        \psi^\nodesymbol$, we have $\src [p] = \left( \src_{[p]} \psi
        \xto{\src_{[p]}} \psi \xto{\sfA} \omega \right)$;
        \item $\tgt \left( \psi \xto{\sfA} \omega \right) = \left( \tgt \psi
        \xto{\tgt} \psi \xto{\sfA} \omega \right)$.
    \end{enumerate}
\end{corollary}
\begin{proof}
    Direct consequence of
    \cref{lemma:polygraphic-realization:generators-of-boundary} and
    \condition{PR1}, \condition{PR2}, and \condition{PR3}.
\end{proof}

\begin{definition}
    \label{def:polygraphic-realization:nabla-boundary:forgetful}
    For $\omega \in \bbOO_n$ and $1 \leq k < n$, we have a morphism $u :
    \nabla_k \partial \polyreal{\omega} \longrightarrow \optPolyFun^{k-1}$
    \[
        \begin{tikzcd}
            \bbOO_{k-1} / \omega
                \ar[d, "u_0"'] &
            E
                \ar[r, "p"]
                \ar[d, "u_2" left]
                \ar[l, "\src" above] &
            \bbOO_k / \omega
                \ar[r, "\tgt"]
                \ar[d, "u_1" left] &
            \bbOO_{k-1} / \omega
                \ar[d, "u_0" left] \\
            \bbOO_{k-1} &
            \bbOO_k^\nodesymbol
                \ar[r, "p"]
                \ar[l, "\src" above] &
            \bbOO_k
                \ar[r, "\tgt"] &
            \bbOO_{k-1}
        \end{tikzcd}
    \]
    induced by the forgetful maps $\bbOO_{k-1} / \omega \longrightarrow
    \bbOO_{k-1}$ and $\bbOO_k / \omega \longrightarrow \bbOO_k$.
\end{definition}

\begin{lemma}
    \label{lemma:polygraphic-realization:lift}
    Let $\omega \in \bbOO_n = \tree \optPolyFun^{n-2}$. The map $\omega :
    \underlyingtree{\omega} \longrightarrow \optPolyFun^{n-2}$ factors through
    $u : \nabla_{n-1} \partial \polyreal{\omega} \longrightarrow
    \optPolyFun^{n-2}$
    (\cref{def:polygraphic-realization:nabla-boundary:forgetful}):
    \[
        \triangleDRdiagram
            {\nabla_{n-1} \partial \polyreal{\omega}}{\underlyingtree{\omega}}
                {\optPolyFun^{n-2} .}
            {\baromega}{u}{\omega}
    \]
\end{lemma}
\begin{proof}
    Let $\baromega$ map a node $[p] \in \omega^\nodesymbol$ to the cell
    $\left(\src_{[p]} \omega \xto{\src_{[p]}} \omega \right) \in \bbOO_{n-1} /
    \omega$, and map an edge $[l]$ to the cell $\left( \edg_{[l]} \omega
    \xto{\edg_{[l]}} \omega \right) \in \bbOO_{n-2} / \omega$ (recall the
    notation from \cref{not:edg,def:o}).
\end{proof}

\begin{proposition}
    \label{prop:polygraphic-realization:parallel}
    On the one hand, consider the tree $\nabla_{n-1} \partial
    \polyreal{\omega}$-tree $\baromega$ of
    \cref{lemma:polygraphic-realization:lift}, and on the other hand, recall
    from \cref{rem:polygraphic-realization:inductive-boundary} that there is a
    $(n-1)$-generator $\left( \tgt \omega \xto{\tgt} \omega \right)$ of
    $\partial \polyreal{\omega}$ corresponding to the target embedding of
    $\omega$. Then, in $\partial \polyreal{\omega}$, the composite
    $\baromega^\circ$ (\cref{def:composition}) and the generator $\left( \tgt
    \omega \xto{\tgt} \omega \right)$ are parallel.
\end{proposition}
\begin{proof}
    If $\omega$ is degenerate, say $\omega = \itree{\phi}$ for some $\phi \in
    \bbOO_{n-2}$, then $\composition{\baromega} = \id_{( \phi \xto{\tgt \tgt}
    \omega )}$, while $\left( \tgt \omega
    \xto{\tgt} \omega \right) = \left( \ytree{\phi}
    \xto{\tgt} \omega \right)$. By \condition{Degen}, those two cells are
    parallel.

    For the rest of the proof, we assume that $\omega$ is not degenerate.
    First, we have
    \[
        \tgt \composition{\baromega}
        \speq \tgt \src_{[]} \baromega
        \speq \tgt \left( \tgt \src_{[]} \omega
            \xto{\tgt \src_{[]}} \omega \right)
        \speq \left( \tgt \tgt \omega \xto{\tgt \tgt} \omega \right)
        \speq \tgt \left( \tgt \omega \xto{\tgt} \omega \right) .
    \]
    Then, in order to show that $\src \composition{\baromega} = \src (\tgt
    \omega \xto{\tgt} \omega)$, we show that the $(n-2)$-generators occurring
    on both sides are the same, and that the way to compose them is unique.
    \begin{enumerate}
        \item Generators in $\src \composition{\baromega}$ are of the form
        $\left( \phi \xto{[q]} \psi \xto{[p]} \omega \right)$, for $[p [q]] \in
        \omega^\leafsymbol$. By \condition{Glob2}, those are equal to
        $\left(\phi \xto{\readdress_\omega [p[q]]} \tgt \omega \xto{\tgt}
        \omega \right)$, which are exactly the generators in the cell $\src
        \left( \tgt \omega \xto{\tgt} \omega \right)$.

        \item To show that there is a unique way to compose all the
        $(n-2)$-generators of the form $(\phi \xto{\src_{[q]}} \psi
        \xto{\src_{[p]}} \omega)$, where $[p[q]]$ ranges over
        $\omega^\leafsymbol$, it is enough to show that no two have the same
        target. Assume $(\phi_i \xto{\src_{[q_i]}} \psi_i \xto{\src_{[p_i]}}
        \omega)$, with $i = 1, 2$, are $(n-2)$-generators occuring in $\src
        \composition{\baromega}$ with the same target. Consider the following
        diagram:
        \[
            \begin{tikzcd} [column sep = large]
                &
                \phi_1
                    \ar[r, "\src_{[q_1]}"]
                    \ar[dr, "\src_{[r_1]}"'] &
                \psi_1
                    \ar[dr, "\src_{[p_1]}"] \\
                \rho
                    \ar[ur, "\tgt"]
                    \ar[dr, "\tgt"'] &
                &
                \tgt \omega
                    \ar[r, "\tgt"] &
                \omega \\
                &
                \phi_2
                    \ar[r, "\src_{[q_2]}"]
                    \ar[ur, "\src_{[r_2]}"] &
                \psi_2
                    \ar[ur, "\src_{[p_2]}"']
            \end{tikzcd}
        \]
        where $[r_i] \eqdef \readdress_\omega [p_i [q_i]] \in \tgt
        \omega^\nodesymbol$. The outer hexagon commutes by assumption, the two
        squares on the right are instances of \condition{Glob2}, and the left
        square commutes as $\tgt : \tgt \omega \longrightarrow \omega$ is a
        monomorphism, since $\omega$ is non degenerate. By inspection of the
        opetopic identities (see \cref{def:o}), the only way for the left
        square to commute is the trivial way, i.e. $[r_1] = [r_2]$. Since
        $\readdress_\omega$ is a bijection, we have $[p_1 [q_1]] = [p_2
        [q_2]]$, thus $[p_1] = [p_2]$ and $[q_1] = [q_2]$.
        \qedhere
    \end{enumerate}
\end{proof}

\begin{definition}
    [Inductive step for $\polyreal{-}$]
    \label{def:polygraphic-realization:inductive}
    For $\omega \in \bbOO_n$, let $\polyreal{\omega}$ be the cellular extension
    \[
        \partial \polyreal{\omega} \xot{\src, \tgt} \left\{ \omega \right\} ,
    \]
    where $\tgt$ maps $\omega$ to the $(n-1)$-generator $\left( \tgt \omega
    \xto{\tgt} \omega \right)$, and where the composition tree of $\src \omega$
    is $\baromega$ (\cref{lemma:polygraphic-realization:lift}). For
    consistency, we also write $\left( \omega \xto{\id} \omega \right)$ for the
    unique $n$-generator of $\polyreal{\omega}$. This is well-defined by
    \cref{prop:polygraphic-realization:parallel}, and gives a functor
    $\polyreal{-} : \bbOO_{\leq n} \longrightarrow \Pol$.
\end{definition}

\begin{proposition}
    \label{prop:polygraphic-realization:inductive}
    For $\omega \in \bbOO_n$, the polygraphs $\partial \polyreal{\omega}$
    (\cref{def:polygraphic-realization:inductive-boundary}) and
    $\polyreal{\omega}$ (\cref{def:polygraphic-realization:inductive}) satisfy
    the \cref{ass:polygraphic-realization}.
\end{proposition}
\begin{proof}
    \begin{itemize}
        \item \condition{PR1} For $j < n$, by
        \cref{lemma:polygraphic-realization:generators-of-boundary}, we already
        have $\polyreal{\omega}_j = \partial \polyreal{\omega}_j = \bbOO_j /
        \omega$. In dimension $n$, the only element of $\bbOO_n / \omega$ is
        $\id : \omega \longrightarrow \omega$, which corresponds to the unique
        $n$-generator of $\polyreal{\omega}$. If $j > n$, then both $\bbOO_j /
        \omega$ and $\polyreal{\omega}_j$ are empty.
        \item \condition{PR2} and \condition{PR3} By definition, those
        hypotheses hold for the unique $n$-generator $\left( \omega \xto{\id}
        \omega \right)$ of $\polyreal{\omega}$. By induction, they also hold on
        the other generators.
        \qedhere
    \end{itemize}
\end{proof}

To conclude, we have defined a functor $\polyreal{-} : \bbOO \longrightarrow
\PolMTO$ which satisfies the \cref{ass:polygraphic-realization} for all $n \in
\bbNN$.

\subsection{The shape function}
\label{sec:shape}

In \cref{sec:equivalence}, we will define an adjunction
$
    \polyreal{-} : \PshO \adjunction \PolMTO : N
$,
where $\polyreal{-} : \PshO \longrightarrow \PolMTO$ is the left Kan extension
of the polygraphic realization defined in \cref{sec:polygraphic-realization},
and $N$ its associated nerve. In order to show that this is an adjoint
\emph{equivalence}, we need a good understanding of $N$. An crucial step is
\cref{th:shape} that establishes a bijection $\catPP_n \longrightarrow
\sum_{\omega \in \bbOO_n} (N \catPP)_\omega$, where $\catPP \in \PolMTO$ and $n
\in \bbNN$. Furthermore, we show that this bijection is \emph{over $\bbOO_n$}:
\[
    \triangleURdiagram
        {\catPP_n}{\sum_{\omega \in \bbOO_n} (N \catPP)_\omega}{\bbOO_n,}
        {\cong}{\shape{(-)}}{}
\]
where the vertical arrow maps an element in the $\omega \in \bbOO_n$ component
to $\omega$, and where $\shape{(-)}$ is the \emph{shape function} which we
define in this section.

Here is a sketch of the definition of $\shape{(-)} : \catPP_n \longrightarrow
\bbOO_n$. The cases $n = 0, 1$ are trivial, since there is a unique $0$-opetope
and a unique $1$-opetope. Assume $n \geq 2$, and take $x \in \catPP_n$. Then
the composition tree of $\src x$ is a tree whose nodes are decorated
$(n-1)$-generators, and edges by $(n-2)$-generators. Replacing these generators
by their inductively defined shapes, we obtain a tree whose nodes are decorated
by $(n-1)$-opetopes, and edges by $(n-2)$-opetopes, a.k.a. an $n$-opetope,
which we shall denote by $\shape{x}$.
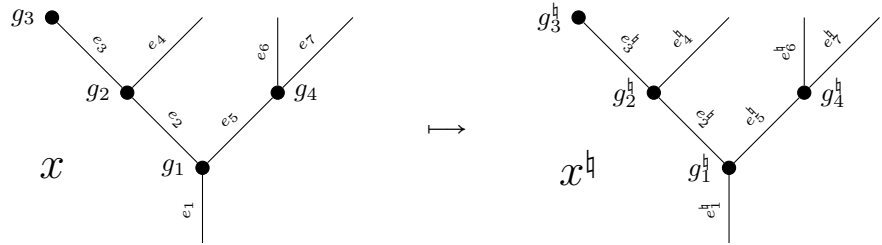
\[
    \tikzinput[.7]{shape}{}
\]
The fact that $\shape{x}$ corresponds to the intuitive notion of ``shape'' of
$x$ is justified by \cref{th:shape}.

\begin{lemma}
    \label{lemma:terminal-mtop:parallel-generators}
    If $x, y \in \catTT_n$ are two parallel generators, then they are equal.
\end{lemma}
\begin{proof}
    This is trivial if $n = 0$, as $\catTT$ only has one $0$-generator. If $n
    \geq 1$, note that $x = (\src x, \tgt x) = (\src y, \tgt y) = y$.
\end{proof}

\begin{proposition}
    \label{prop:terminal-mtop:opetopes}
    For $x \in \catTT_n$ there exists a unique $\shape{x} \in \bbOO_n$ such
    that the terminal morphism $\shriek : \polyreal{\shape{x}} \longrightarrow
    \catTT$ maps $\shape{x}$ (the unique $n$-generator of
    $\polyreal{\shape{x}}$) to $x$. In particular, the map $\shape{(-)} :
    \catTT_n \longrightarrow \bbOO_n$ is a bijection.
\end{proposition}
\begin{proof}
    \begin{itemize}
        \item (Uniqueness) Assume that there exists two distinct opetopes
        $\omega, \omega' \in \bbOO_k$ such that $\shriek \omega = \shriek
        \omega'$, with $k$ minimal for this property. Then necessarily, $k \geq
        2$. On the one hand, we have $\underlyingtree{\omega} =
        \underlyingtree{\ct \src \omega} = \underlyingtree{\ct \src \omega'} =
        \underlyingtree{\omega'}$. On the other hand, for $[p] \in
        \omega^\nodesymbol = (\omega')^\nodesymbol$, we have
        \begin{align*}
            \shriek \src_{[p]} \omega
            &\speq \shriek \src_{[p]} \omega
                & \text{since $\shriek$ is also }
                \polyreal{\src_{[p]} \omega} \hookrightarrow \polyreal{\omega} \rightarrow \catTT \\
            &\speq \src_{[p]} \shriek \omega
                & \text{since $\shriek$ is a morphism of polygraphs} \\
            &\speq \src_{[p]} \shriek \omega'
                & \text{by assumption} \\
            &\speq \shriek \src_{[p]} \omega'
                & \text{since $\shriek$ is a morphism of polygraphs} \\
            &\speq \shriek \src_{[p]} \omega'
                & \text{since $\shriek$ is also }
                \polyreal{\src_{[p]} \omega'} \hookrightarrow \polyreal{\omega'}
                \rightarrow \catTT,
        \end{align*}
        and by minimality of $k$, we have $\src_{[p]} \omega = \src_{[p]}
        \omega'$, for any address $[p]$. Consequently, $\omega = \omega'$, a
        contradiction.

        \item (Existence) The cases $n = 0, 1$ are trivial, so assume $n \geq
        2$, and that by induction, the result holds for all $k < n$, i.e. that
        for $g \in \catPP_k$, there is a unique opetope $\shape{g} \in \bbOO_k$
        such that $\shriek (\shape{g}) = g$. In particular the following
        two triangles commute:
        \[
            \diagramsize{2}{3}
            \triangleURdiagram
                {\polyreal{\src_{[p]} \shape{g}}}{\polyreal{\shape{g}}}{\catPP,}
                    {\polyreal{\src_{[p]}}}
                {\shriek}{\shriek}
            \qqquad
            \triangleURdiagram
                {\polyreal{\tgt \shape{g}}}{\polyreal{\shape{g}}}{\catPP ,}
                {\polyreal{\tgt}}{\shriek}{\shriek}
        \]
        where $[p] \in (\shape{g})^\nodesymbol$. Consequently,
        $\shape{(\src_{[p]} g)} = \src_{[p]} (\shape{g})$ and $\shape{(\tgt g)}
        = \tgt (\shape{g})$, and the following displays an isomorphism
        $\nabla_{n-1} \catTT \longrightarrow \optPolyFun^{n-2}$:
        \[
            \begin{tikzcd}
                \catTT_{n-2}
                    \ar[d, "\shape{(-)}"'] &
                \catTT_{n-1}^\nodesymbol
                    \ar[r, "p"]
                    \ar[d]
                    \ar[l, "\src" above] &
                \catTT_{n-1}
                    \ar[r, "\tgt"]
                    \ar[d, "\shape{(-)}" left] &
                \catTT_{n-2}
                    \ar[d, "\shape{(-)}" left] \\
                \bbOO_{n-2} &
                \bbOO_{n-1}^\nodesymbol
                    \ar[r, "p"]
                    \ar[l, "\src" above] &
                \bbOO_{n-1}
                    \ar[r, "\tgt"] &
                \bbOO_{n-2} .
            \end{tikzcd}
        \]
        Hence, the composite $\underlyingtree{\ct \src x} \xto{\ct \src x}
        \nabla_{n-1} \catTT \xto{\shape{(-)}} \optPolyFun^{n-2}$ defines an
        $n$-opetope $\shape{x}$ with $\underlyingtree{\shape{x}} =
        \underlyingtree{\ct \src x}$. We claim that $\shriek (\shape{x}) = x$. We
        first show that $\shriek \src (\shape{x}) = \src x$. We have
        \begin{align*}
            \underlyingtree{\ct \src x}
            &\speq \underlyingtree{\shape{x}}
                & \text{by definition} \\
            &\speq \underlyingtree{\ct \src (\shape{x})}
                & \text{by \condition{PR3}} \\
            &\speq \underlyingtree{\ct \shriek \src (\shape{x})}
                & \text{since $\shriek$ is a morphism of polygraphs.}
        \end{align*}
        Then, for any address $[p]$ in $\underlyingtree{\ct \src x}$, we have
        \begin{align*}
            \src_{[p]} x
            &\speq \shriek (\shape{(\src_{[p]} x)})
                & \text{by induction} \\
            &\speq \shriek \src_{[p]} (\shape{x})
                & \text{by definition of $\shape{x}$} \\
            &\speq \src_{[p]} \shriek (\shape{x})
                & \text{since $\shriek$ is a morphism of polygraphs,}
        \end{align*}
        and therefore, by \cref{prop:composition-tree-duality}, $\src x = \src
        \shriek (\shape{x})$. Next,
        \begin{align*}
            \tgt \shriek (\shape{x})
            &\speq \shriek \tgt (\shape{x})
                & \text{by induction} \\
            &\:\parallel\: \shriek \src (\shape{x})
                & \text{since } \src (\shape{x}) \parallel \tgt (\shape{x}) \\
            &\speq \src x
                & \text{showed above} \\
            &\:\parallel\: \tgt x,
        \end{align*}
        and therefore, $\tgt \shriek (\shape{x}) = \tgt x$. Finally, $\shriek
        (\shape{x}) \parallel x$, and by
        \cref{lemma:terminal-mtop:parallel-generators}, $\shriek (\shape{x}) = x$.
        \qedhere
    \end{itemize}
\end{proof}

\begin{notation}
    In the light of \cref{prop:terminal-mtop:opetopes}, we identify $\catTT_n$
    with $\bbOO_n$. This identification is compatible with faces, i.e.
    $\src_{[p]}$ and $\tgt$. Further, $\shriek : \polyreal{\omega}
    \longrightarrow \catTT$ maps a generator $(\phi \rightarrow \omega)$ to
    $\phi$.
\end{notation}

\begin{notation}
    \label{not:polygraphes-opetopes}
    For $\catPP \in \PolMTO$ and $\omega \in \bbOO_n$, let $\catPP_\omega
    \eqdef \left\{x \in \catPP_n \mid \shape{x} = \omega \right\}$. If $f :
    \catPP \longrightarrow \catQQ$ is a morphism of polygraphs, then it
    restricts and corestricts as a map $f : \catPP_\omega \longrightarrow
    \catQQ_\omega$.
\end{notation}

\begin{theorem}
    \label{th:shape}
    For $\catPP \in \PolMTO$ and $x \in \catPP_n$, there exists a unique
    pair\footnote{In \cite[proposition 2.2.3 (2)]{Henry2019}, $\shape{x}$ is
    written $\uX$ and called the \emph{universal cell} (or \emph{top cell}) of
    $x$.}
    \[
        \left( \shape{x}, \polyreal{\shape{x}} \xto{\tildX} \catPP \right)
        \:\in\: \polyreal{-} / \catPP
    \]
    such that $\tildX_n (\shape{x}) = x$. Further, the \emph{shape function}
    $\shape{(-)} : \calPP_n \longrightarrow \bbOO_n$ maps an $n$-generator $x$
    to $\shape{x} = \shriek x$, and the map
    \begin{equation}
        \label{eq:shape-tilde}
        \widetilde{(-)} : \catPP_\omega
            \longrightarrow \PolMTO (\polyreal{\omega}, \catPP)
    \end{equation}
    is a bijection\footnote{In other words, the functor $\PolMTO
    \longrightarrow \Set$ that maps a polygraph $\catPP$ to $\catPP_n$ is
    \emph{familially representable}
    (\cref{def:effective-and-familially-representable}) with
    $\left\{\polyreal{\omega} \mid \omega \in \bbOO_n \right\}$ as representing
    family.}.
\end{theorem}
\begin{proof}
    \begin{itemize}
        \item (Uniqueness) Assume $\polyreal{\omega} \xto{f} \catPP \xot{f'}
        \polyreal{\omega'}$ are different morphisms such that $f (\omega) = x =
        f' (\omega')$. Then $\shriek \omega = \shriek f (\omega) = \shriek f' (\omega') =
        \shriek \omega'$, and by \cref{prop:terminal-mtop:opetopes}, $\omega =
        \omega'$. Let $\left(\phi \xto{\sfA} \omega \right) \in
        \polyreal{\omega}_k$ be such that $f
        \left( \phi \xto{\sfA} \omega \right) \neq f' \left(\phi \xto{\sfA}
        \omega \right)$, with $k$ minimal for this property. Then $k < n$
        (since by assumption $f (\omega) = x = f' (\omega')$), and $\sfA$
        factorizes as $\left(\phi \xto{\sfJ} \psi \xto{\sfB} \omega \right)$,
        where $\sfJ$ is a face embedding, i.e. either $\tgt$ or $\src_{[p]}$
        for some $[p] \in \omega^\nodesymbol$. Then by assumption,
        \begin{align*}
            f \left( \phi \xto{\sfA} \omega \right)
            &\speq \sfJ f \left(
                \psi \xto{\sfB} \omega \right) \\
            &\speq \sfJ f' \left(
                \psi \xto{\sfB} \omega \right)
                & \text{by minimality of $k$} \\
            &\speq f' \left( \phi \xto{\sfA} \omega \right) ,
        \end{align*}
        a contradiction.

        \item (Existence) The cases $n = 0, 1$ are trivial, so assume $n \geq
        2$, and that by induction, the result holds for all $k < n$. Let
        $\shape{x} \eqdef \shriek x \in \bbOO_n$. We wish to construct a morphism
        $O [\shape{x}] \xto{\tildX} \catPP$ having $x$ in its image. For
        $\left(\psi \xto{\sfJ} \shape{x} \right)$ a face of $\shape{x}$ (i.e.
        $\tgt$ or $\src_{[p]}$ for some $[p] \in (\shape{x})^\nodesymbol$), we
        have $\shape{(\operatorname{\sfJ} x)} = \psi$, so that by induction,
        there exists a morphism $\polyreal{\psi}
        \xto{\widetilde{\operatorname{\sfJ} x}} \catPP$ having
        $\operatorname{\sfJ} x$ in its image, providing a commutative square
        \[
            \squarediagram
                {\polyreal{\psi}}{\catPP}{\polyreal{\shape{x}}}{\catTT .}
                {\widetilde{\sfJ x}}{\polyreal{\sfJ}}{\shriek}{\shriek}
        \]
        To alleviate upcoming notations, write $\bar{\sfJ} \eqdef
        \widetilde{\operatorname{\sfJ} x} : \polyreal{\psi} \longrightarrow
        \catPP$. Let $\left( \phi \xto{\sfA} \shape{x} \right) \in \bbOO_{< n}
        / \shape{x}$. If $\sfA$ is a face embedding, define $\bar{\sfA}$ as
        before. If not, then it factors through a face embedding as $\sfA =
        \left( \phi \xto{\sfJ} \psi \xto{\sfB} \omega \right)$, and let
        $\bar{\sfA} \eqdef \bar{\sfB} \cdot \polyreal{\sfJ}$. Then the left
        square commutes, and passing to the colimit over $\bbOO_{< n} /
        \shape{x}$, we obtain the right square:
        \[
            \squarediagram
                {\polyreal{\phi}}{\catPP}{\polyreal{\shape{x}}}{\catTT ,}
                {\bar{\sfA}}{\polyreal{a}}{\shriek}{\shriek}
            \qqquad
            \squarediagram
                {\partial \polyreal{\shape{x}}}{\catPP}{\polyreal{\shape{x}}}{\catTT .}
                {f}{}{\shriek}{\shriek}
        \]
        We want a diagonal filler of the right square. Since
        $\polyreal{\shape{x}}$ is a one-generator cellular extension of
        $\partial \polyreal{\shape{x}}$
        (\cref{def:polygraphic-realization:inductive}), it is enough to check
        that $f \src \shape{x} = \src x$, and $f \tgt \shape{x} = \tgt x$. The
        latter is clear, as $f$ extends $\bar{\tgt} : \polyreal{\tgt
        \shape{x}} \longrightarrow \catPP$, and $f \tgt \shape{x} = \bar{\tgt}
        \tgt \shape{x} = \tgt x$ by definition. We now proceed to prove the
        former. First, $\underlyingtree{\ct \src \shape{x}} =
        \underlyingtree{\ct \src x}$ since both are mapped to the same element
        of $\catTT_n$. Then, for $[p]$ a node address of $\ct \src \shape{x}$,
        we have $f \src_{[p]} \shape{x} = \overline{\src_{[p]}} \src_{[p]}
        \shape{x} = \src_{[p]} x$. Hence $f \src \shape{x} = \src x$.
        \qedhere
    \end{itemize}
\end{proof}

\subsection{The adjoint equivalence}
\label{sec:equivalence}

\begin{definition}
    [Polygraphic realization-nerve adjunction]
    \label{def:polygraphic-realization-nerve}
    The polygraphic realization functor $\polyreal{-} : \bbOO \longrightarrow
    \PolMTO$ extends to a left adjoint
    \[
        \polyreal{-} : \PshO \adjunction \PolMTO : N ,
    \]
    by left Kan extension of $\polyreal{-} : \bbOO \longrightarrow \PolMTO$
    along the Yoneda embedding $\yoneda : \bbOO \longrightarrow \PshO$.
    Explicitly, the polygraphic realization of an opetopic set $X \in \PshO$
    can be computed with the coend on the left, while the \emph{polygraphic
    nerve} $N \catPP$ of a polygraph $\catPP \in
    \PolMTO$ is given on the right:
    \[
        \polyreal{X}
        \speq \int^{\omega \in \bbOO} X_\omega \times \polyreal{\omega} ,
        \qqquad
        N \catPP
        \speq \PolMTO (\polyreal{-}, \catPP) : \bbOO^\op \longrightarrow \Set .
    \]
\end{definition}

\begin{theorem}
    \label{th:opetopic-sets-mto-polygraphs}
    The unit and counit are natural isomorphisms. Consequently, the polygraphic
    realization-nerve adjunction of \cref{def:polygraphic-realization-nerve} is
    an adjoint equivalence between $\PshO$ and $\PolMTO$.
\end{theorem}
\begin{proof}
    Let $X \in \PshO$ and $\catPP \in \PolMTO$.  Note that with the nerve
    functor of \cref{def:polygraphic-realization-nerve}, the bijection of
    \cref{eq:shape-tilde} becomes
    $
        \widetilde{(-)} : \catPP_\omega \longrightarrow (N \catPP)_\omega
    $.
    An $n$-generator $x \in \catPP_n$ then corresponds to a cell $\tildX \in N
    \catPP_\omega$, where $\omega \eqdef \shape{x}$, in other words, the shape
    function partitions the set of $n$-generators to form an opetopic set $N
    \catPP$. In the converse direction, for $X \in \PshO$,
    \begin{align*}
        \polyreal{X}_n
        &\speq \int^{\omega \in \bbOO} X_\omega \times \polyreal{\omega}_n
            & \text{by \cref{def:polygraphic-realization-nerve}} \\
        &\spcong \int^{\omega \in \bbOO} X_\omega \times \bbOO_n / \omega
            & \text{see \cref{ass:polygraphic-realization}, \condition{PR1}} \\
        &\spcong \int^{\omega \in \bbOO}
            X_\omega \times \sum_{\psi \in \bbOO_n} \bbOO (\psi, \omega) \\
        &\spcong \sum_{\psi \in \bbOO_n} \int^{\omega \in \bbOO}
            X_\omega \times \bbOO (\psi, \omega) \\
        &\spcong \sum_{\psi \in \bbOO_n} X_\psi ,
    \end{align*}
    so the set of $n$-generators of $\polyreal{X}$ is the set of $n$-cells of
    $X$. In particular, $\polyreal{X}_\omega \cong X_\omega$. Finally, the unit
    $\eta : X \longrightarrow N \polyreal{X}$ and counit $\epsilon :
    \polyreal{N \catPP} \longrightarrow \catPP$ are the following composites
    \[
        X_\omega
        \xto{\cong} \polyreal{X}_\omega
        \xto{\widetilde{(-)}} \left( N \polyreal{X} \right)_\omega ,
    \]
    \[
        \polyreal{N \catPP}_n
        \xto{=} \sum_{\omega \in \bbOO_n} \polyreal{N \catPP}_\omega
        \xto{\cong} \sum_{\omega \in \bbOO_n} (N \catPP)_\omega
        \xto{\widetilde{(-)}^{-1}} \sum_{\omega \in \bbOO_n} \catPP_\omega
        \xto{=} \catPP_n ,
    \]
    which by definition are isomorphisms.
\end{proof}

Many-to-one polygraphs have been the subject of other work \cite{Harnik2002}
\cite{Harnik2008}, and proved to be equivalent to the notion of
\emph{multitopic sets}. This, together with
\cref{th:opetopic-sets-mto-polygraphs}, prove the following:

\begin{corollary}
    The category $\PshO$ of opetopic sets is equivalent to the category of
    multitopic sets.
\end{corollary}

An \emph{opetopic plex} is an opetopic polyplex of the form $\uU$, where $u \in
\catTT_n$ (as opposed to $\catTT^*_n$). In \cite[corollary 2.4.9 and remark
2.5.1]{Henry2019}, Henry shows that $\PolMTO$ is a presheaf category over some
category $\Oplex$ of opetopic plexes, and asks wether they are the same as
opetopes. We now answer this question positively.

\begin{definition}
    [Cauchy-complete category]
    \label{def:cauchy-complete}
    An idempotent morphism $e : a \longrightarrow a$ \emph{splits} if it
    decomposes as $e = i r$ with $r i = \id_a$. A category is
    \emph{Cauchy-complete} if all its idempotent morphisms split.
\end{definition}

\begin{theorem}
    [{\cite[theorem 1]{Borceux1986}}]
    \label{th:cauchy-complete}
    Let $\catAA$ and $\catBB$ be Cauchy-complete categories. An equivalence of
    categories $\Psh\catAA \xto{\simeq} \Psh\catBB$ restricts and corestricts
    to the representable presheaves, or in other words, to an equivalence
    $\catAA \xto{\simeq} \catBB$.
\end{theorem}

\begin{corollary}
    \label{coroll:opt-oplex}
    The category $\bbOO\mathrm{plex}$ of opetopic plexes is equivalent to
    $\bbOO$.
\end{corollary}
\begin{proof}
    By definition, $\bbOO$ is a directed category, and by \cite[proposition
    2.2.3 (4)]{Henry2019}, so is $\bbOO\mathrm{plex}$. In particular, they are
    both Cauchy-complete. On the other hand, $\PshO \simeq \PolMTO \simeq
    \Psh{\bbOO\mathrm{plex}}$, and we conclude using \cref{th:cauchy-complete}.
\end{proof}

In \cite{Palm2004}, Palm studies another approach to weak higher-dimensional
categories, based on \emph{dendrotopic sets}, and show that dendrotopic sets
are equivalent to many-to-one polygraphs. Therefore,

\begin{corollary}
    \label{coroll:opt-dendro}
    Opetopic sets are equivalent to dendrotopic sets.
\end{corollary}

\section{Conclusion}

To conclude, here is a diagram of the various approaches to the notion of
pasting diagrams and ``many-to-one shapes'', and how they are related
throughout the litterature and in this paper:

\begin{figure} [ht]
    \tikzinput{misc-diagrams}{mto-polygraphs}
\end{figure}

\bibliographystyle{alpha}
\bibliography{bibliography}

\begin{thebibliography}{KJBM10}

\bibitem[BD86]{Borceux1986}
Francis Borceux and Dominique Dejean.
\newblock Cauchy completion in category theory.
\newblock {\em Cahiers Topologie G\'{e}om. Diff\'{e}rentielle Cat\'{e}g.},
  27(2):133--146, 1986.

\bibitem[BD98]{Baez1998}
John~C. Baez and James Dolan.
\newblock Higher-dimensional algebra. {III}. {$n$}-categories and the algebra
  of opetopes.
\newblock {\em Advances in Mathematics}, 135(2):145--206, 1998.

\bibitem[Che03]{Cheng2003}
Eugenia Cheng.
\newblock The category of opetopes and the category of opetopic sets.
\newblock {\em Theory and Applications of Categories}, 11:No. 16, 353--374,
  2003.

\bibitem[Che04a]{Cheng2004a}
Eugenia Cheng.
\newblock Weak {$n$}-categories: comparing opetopic foundations.
\newblock {\em Journal of Pure and Applied Algebra}, 186(3):219--231, 2004.

\bibitem[Che04b]{Cheng2004}
Eugenia Cheng.
\newblock Weak {$n$}-categories: opetopic and multitopic foundations.
\newblock {\em Journal of Pure and Applied Algebra}, 186(2):109--137, 2004.

\bibitem[Che13]{Cheng2013}
Eugenia Cheng.
\newblock A direct proof that the category of 3-computads is not {C}artesian
  closed.
\newblock {\em Cahiers de Topologie et G\'eom\'etrie Diff\'erentielle
  Cat\'egoriques}, 54(1):3--12, 2013.

\bibitem[CJ95]{Carboni1995}
Aurelio Carboni and Peter Johnstone.
\newblock Connected limits, familial representability and {A}rtin glueing.
\newblock {\em Mathematical Structures in Computer Science}, 5(4):441--459,
  1995.
\newblock Fifth Biennial Meeting on Category Theory and Computer Science
  (Amsterdam, 1993).

\bibitem[CJ04]{Carboni2004}
Aurelio Carboni and Peter Johnstone.
\newblock Corrigenda for: {C}onnected limits, familial representability and
  {A}rtin glueing.
\newblock {\em Mathematical Structures in Computer Science}, 14(1):185--187,
  2004.

\bibitem[GK13]{Gambino2013}
Nicola Gambino and Joachim Kock.
\newblock Polynomial functors and polynomial monads.
\newblock {\em Mathematical Proceedings of the Cambridge Philosophical
  Society}, 154(1):153--192, 2013.

\bibitem[GM09]{Guiraud2009}
Yves Guiraud and Philippe Malbos.
\newblock Higher-dimensional categories with finite derivation type.
\newblock {\em Theory and Applications of Categories}, 22:No. 18, 420--478,
  2009.

\bibitem[Hen19]{Henry2019}
Simon Henry.
\newblock Non-unital polygraphs form a presheaf category.
\newblock {\em Higher Structures}, 3(1):248--291, 2019.

\bibitem[HMP00]{Hermida2000}
Claudio Hermida, Michael Makkai, and John Power.
\newblock On weak higher dimensional categories. {I}. 1.
\newblock {\em Journal of Pure and Applied Algebra}, 154(1-3):221--246, 2000.
\newblock Category theory and its applications (Montreal, QC, 1997).

\bibitem[HMP02]{Hermida2002}
Claudio Hermida, Michael Makkai, and John Power.
\newblock On weak higher-dimensional categories. {I}. 3.
\newblock {\em Journal of Pure and Applied Algebra}, 166(1-2):83--104, 2002.

\bibitem[HMZ02]{Harnik2002}
Victor Harnik, Michael Makkai, and Marek Zawadowski.
\newblock Multitopic sets are the same as many-to-one computads.
\newblock Available at
  \url{http://www.math.mcgill.ca/makkai/m_1_comp/Manytoonecomputads.pdf}, 2002.

\bibitem[HMZ08]{Harnik2008}
Victor Harnik, Michael Makkai, and Marek Zawadowski.
\newblock Computads and multitopic sets.
\newblock {\em arXiv e-prints}, 2008.

\bibitem[KJBM10]{Kock2010}
Joachim Kock, Andr\'e Joyal, Michael Batanin, and Jean-Fran\c{c}ois Mascari.
\newblock Polynomial functors and opetopes.
\newblock {\em Advances in Mathematics}, 224(6):2690--2737, 2010.

\bibitem[Koc11]{Kock2011}
Joachim Kock.
\newblock Polynomial functors and trees.
\newblock {\em International Mathematics Research Notices}, 2011(3):609--673,
  January 2011.

\bibitem[Lei04]{Leinster2004}
Tom Leinster.
\newblock {\em Higher Operads, Higher Categories}.
\newblock Cambridge University Press, 2004.

\bibitem[M\'03]{Metayer2003}
Fran\c{c}ois M\'etayer.
\newblock Resolutions by polygraphs.
\newblock {\em Theory and Applications of Categories}, 11:No. 7, 148--184,
  2003.

\bibitem[MZ08]{Makkai2008}
Michael Makkai and Marek Zawadowski.
\newblock The category of 3-computads is not {C}artesian closed.
\newblock {\em Journal of Pure and Applied Algebra}, 212(11):2543--2546, 2008.

\bibitem[Pal04]{Palm2004}
Thorsten Palm.
\newblock Dendrotopic sets.
\newblock In {\em Galois theory, {H}opf algebras, and semiabelian categories},
  volume~43 of {\em Fields Inst. Commun.}, pages 411--461. American
  Mathematical Society, Providence, RI, 2004.

\bibitem[Str76]{Street1976}
Ross Street.
\newblock Limits indexed by category-valued {$2$}-functors.
\newblock {\em Journal of Pure and Applied Algebra}, 8(2):149--181, 1976.

\end{thebibliography}

\end{document}